\documentclass{amsart}[10pt]
\usepackage[utf8]{inputenc}
\usepackage[english]{babel}
\usepackage[T1]{fontenc}
\usepackage{amsmath}
\usepackage{amsfonts}
\usepackage{amssymb}
\usepackage{amsthm}
\usepackage{amscd}
\usepackage{geometry}
\usepackage{enumitem} 
\usepackage{cancel} 
\usepackage{graphicx}
\usepackage{mathtools}
\usepackage{color}
\usepackage{empheq}
\usepackage{hyperref}
\usepackage{comment}

\newcommand\R{\mathbb{R}}
\newcommand\C{\mathbb{C}}

\newcommand{\enstq}[2]{\left\{#1~\middle|~#2\right\}}
\newcommand\Div{\mathrm{div}}

\newcommand\eps{\varepsilon}
\renewcommand{\Re}{\operatorname{Re}}
\renewcommand{\Im}{\operatorname{Im}}
\newcommand\kin{\mathrm{kin}}
\newcommand\ent{\mathrm{ent}}
\newcommand\loc{\mathrm{loc}}

\theoremstyle{plain}
\newtheorem{theorem}{Theorem} [section]
\newtheorem{lemma}[theorem]{Lemma}

\newtheorem{proposition}[theorem]{Proposition}

\theoremstyle{remark}
\newtheorem{remark}[theorem]{Remark}

\theoremstyle{definition}
\newtheorem{definition}[theorem]{Definition}

\newtheorem{assumption}[theorem]{Assumption}

\numberwithin{equation}{section}

\title{Dynamics of the Schrödinger-Langevin equation}

\author{Quentin Chauleur}
\address{Univ Rennes, CNRS, IRMAR - UMR 6625, F-35000 Rennes, France}
\email{quentin.chauleur@ens-rennes.fr}

\begin{document}

\maketitle

\begin{abstract}
We consider the nonlinear Schrödinger-Langevin equation for both signs of the logarithmic nonlinearity. We explicitly compute the dynamics of Gaussian solutions for large times, which is obtained through the study of a particular nonlinear differential equation of order 2. We then give the asymptotic behavior of general energy weak solutions under some regularity assumptions. Some numerical simulations are performed in order to corroborate the theoretical results.
\end{abstract}

\tableofcontents

\section{Introduction}

We consider the following nonlinear Schrödinger equation:
\begin{equation}  i  \partial_t \psi + \frac{1}{2}  \Delta \psi = \lambda \psi \log(|\psi|^2) + \frac{1}{2i}  \mu  \psi \log\left( \frac{ \psi}{\psi^*} \right),  \label{NMeq} \end{equation} 
with $\psi(0,x)=\psi_0(x)$, $x \in \R^d$, $t \in \R$, $\mu > 0$ and $\lambda \in \R^*$. This equation first appears in Nassar's paper \cite{nassar1985} as a possible way to give a stochastic interpretation of quantum mechanics in the context of Bohmian mechanics. It had a recent renewed interest in the physics community, in particular in quantum mechanics in order to describe the continuous measurement of the position of a quantum particle (see for example \cite{nassar}, \cite{zander} or \cite{mousavi2019}) and in cosmology and statistical mechanics (see \cite{chavanis2017}, \cite{chavanis2019cosmo} or \cite{chavanis2019stat}). Note that in its physical interpretation, $\lambda=2 k_B T/ \hbar$ corresponds to a quantum friction coefficient, so both positive and negative signs could be of interest ($k_B$ and $\hbar$ denotes respectively the Boltzmann and the normalized Planck constant, and $T$ is an effective temperature), unlike the real friction coefficient $\mu$ which is taken positive (see \cite{chavanis2017}).

The nonlinear potential $\frac{1}{2i}\log\left( \psi/\psi^* \right)$, which could be seen as the argument of the wave function $\psi$, is not defined when $\psi$ is equal to zero. In order to circumvent this difficulty, we will rather consider the fluid formulation of equation \eqref{NMeq}. Plugging the Madelung transform  $\psi = \sqrt{\rho}e^{iS/\eps}$, or in a more rigorous way the change of unknown functions $\rho = |\psi|^2$ and $J= \Im(\psi^* \nabla \psi)$, we get the following quantum hydrodynamical system

\begin{subequations} \label{hydro}
\begin{empheq}[left=\empheqlbrace]{align}
& \partial_t \rho + \Div J =0 \label{continuity}  \\
& \partial_t J + \Div \left( \frac{J \otimes J}{\rho} \right) + \lambda \nabla \rho + \mu J= \frac{1}{2} \rho \nabla \left( \frac{\Delta \sqrt{\rho}}{\sqrt{\rho}} \right), \label{fluid}
\end{empheq} 
\end{subequations}
where 
\[\Div \left( \frac{J \otimes J}{\rho} \right) = \left( \sum_{j=1}^d \frac{\partial}{\partial x_j} \left( \frac{ J_i J_j}{\rho} \right)    \right)_{1 \leq i \leq d} \]
denoting $J=(J_i)_{1 \leq i \leq d}$.
We note that the last nonlinear term of equation \eqref{fluid}, usually called the Bohm potential, could also be expressed as
\begin{equation} \label{bohm_potential}    
 \frac{1}{2} \rho \nabla \left( \frac{\Delta \sqrt{\rho}}{\sqrt{\rho}} \right)= \frac{1}{4} \Delta \nabla \rho - \Div ( \nabla \sqrt{\rho} \otimes \nabla \sqrt{\rho})    
 \end{equation}
under some regularity assumptions, and its study is at the heart of the Bohmian dynamics theory (see \cite{teufel2009}, \cite{sparber2014} and \cite{wyatt2003}). 

Equations \eqref{continuity}-\eqref{fluid} stand as an isothermal fluid system with an additional dissipation term. In \cite{carles_rigidity}, the study of the isothermal Euler-Korteweg in the case $\lambda >0$ and $\mu = 0$ reveals an asymptotic vanishing Gaussian profile for every rescaled global weak solution with well-prepared initial data. This property, rather unusual in the hydrodynamic setting, is a direct consequence of the link with the defocusing logarithmic Schrödinger equation \cite{carles2018}, namely 
\begin{equation} \label{logNLS}
 i  \partial_t \psi + \frac{1}{2}  \Delta \psi = \lambda \psi \log(|\psi|^2).  
\end{equation} 
On the contrary, in the focusing case (which is way more studied in the mathematical literature, see \cite{cazenave} or \cite{ardila2016}), the existence and stability of periodic non-dispersive solutions, called solitons, is proved and well-understood. The existence of shape-moving periodic solutions, called \textit{breathers}, is also known \cite{ferriere2019}, but no stability results has yet been proved. Let us finally note that the potential energy of this equation,
\[ \lambda \int_{\R^d} |\psi|^2 \log |\psi|^2, \]
has no definite sign regardless of the sign of $\lambda$, so the focusing or defocusing behavior of its solutions may not be clear at first glance. 

The question of the existence of solutions to this kind of quantum system is already dealt with in the case of barotropic pressure of the form $P(\rho)= \lambda \rho^{\gamma}$ in \cite{antonelli2009}, where $\gamma >1$. However, the proof is based on the link with the power-like Schrödinger equation 
\begin{equation} \label{powerNLS}
i  \partial_t \psi + \frac{1}{2}  \Delta \psi = \lambda \psi |\psi|^{\gamma-1}  
\end{equation}
and the use of Strichartz estimates which do not seem to be helpful for the logarithmic nonlinearity. The global existence of weak solutions of equations \eqref{continuity}-\eqref{fluid} will be the aim of a forthcoming paper. We will concentrate here on the study of the large-time behavior of solutions.

This paper is organized as follows. In Section 2, we provide energy estimates, assumptions about existence and regularity of solutions of equations \eqref{continuity}-\eqref{fluid}, and we state the main results of this paper. In Section 3, we explicitly compute the behavior of Gaussian solutions of \eqref{NMeq} for both the focusing ($\lambda <0$) and the defocusing ($\lambda >0$) case, which will be crucial to the study of the universal dynamics of our solutions. In Section 4 we prove the theorems stated in Section 2. Section 5 is devoted to the simulation of numerical trajectories of solutions of \eqref{NMeq}, in order to illustrate our results and comfort our initial assumptions.

\section{First properties and main results}

Equation \eqref{continuity} formally induces some mass conservation for all $t \geq 0$:
\begin{equation} \label{mass_conservation}
 \| \rho(t)\|_{L^1(\R^d)} = \| \rho_0 \|_{L^1(\R^d)}
\end{equation}

Note that system \eqref{continuity}-\eqref{fluid} can also be written in terms of unknown function $u=J/\rho$, which is more convenient in order to express the following energy estimate, which holds for all $t \geq 0$:
\begin{equation}  E(t) +  \mu \int_0^t D(s) ds \leq E_0, \label{eq_energy} \end{equation}
where
\begin{gather}
    E(t)=E(\rho,u)= E_c(t)+E_p(t), \label{energy} \\ 
    E_c(t) = \frac{1}{2} \int_{\R^d}  \rho |u|^2  dx ,\label{energy_cin}\\
    E_p(t)=  \int_{\R^d} \left( \frac{1}{2} |\nabla \sqrt{\rho} |^2 + \lambda \rho\log \rho \right)dx,\label{energy_pot} \\
    D(s)=D(\rho,u) =  \int_{\R^d} \rho |u|^2 dx = 2 E_c(s), \label{dissipation}
\end{gather}
and $E_0:=E(0)$. We note that equation \eqref{NMeq} has a conserved $L^2$-norm but a dissipated energy, which is a typical feature in the context of fluid dynamics but a quite unusual one in the framework of nonlinear Schrödinger equations. In particular, for the focusing case $\lambda <0$, the dissipation term in estimate \eqref{eq_energy} implies that $E_c \in L^1(\R)$. In fact, this is a direct consequence of the mass conservation \eqref{mass_conservation} and the logarithmic Sobolev inequality 
\begin{equation} \label{log_sobo_ineq}
 \int \rho \log \rho \leq \frac{\alpha^2}{\pi} \| \nabla \sqrt{\rho} \|^2_{L^2} + \left(\log \| \rho \|_{L^1} - d \left(1+\log \alpha \right) \right) \| \rho \|_{L^1}  
\end{equation}
which holds for any $\alpha >0$ and for any function $\rho \in W^{1,1}(\R^d) \backslash\left\{ 0 \right\}$ (for a proof of this inequality we refer to \cite{lieb_loss}). Note that $E_c \in L^1(\R)$ still holds in the defocusing case $\lambda >0$ under some regularity assumptions (see Remark \ref{rem_defocusing} below).\\ 
A first consequence of this important property is that the center of mass of every solution converges, which could be seen through the formal integration by part of equation \eqref{fluid} (using the alternative formulation of the Bohm potential \eqref{bohm_potential}),
\[ \frac{d}{dt} \left(\int_{\R^d}  \rho(t,x) u(t,x) dx \right)= \int_{\R^d}  \partial_t (\rho(t,x) u(t,x)) dx = - \mu  \int_{\R^d}  \rho(t,x) u(t,x) dx,  \]
and of equation \eqref{continuity},
\begin{equation*} 
\frac{d}{dt} \left(\int_{\R^d} x \rho(t,x) dx \right)= \int_{\R^d} x  \partial_t \rho(t,x) dx = - \int_{\R^d} \rho(t,x) u(t,x) dx = - e^{- \mu t} \int_{\R^d} \rho_0(x) u_0(x) dx,
\end{equation*} 
so finally
\begin{equation} \label{center_mass_conv}
\int_{\R^d} x \rho(t,x) dx \underset{t\to+\infty}{\rightarrow} \int_{\R^d} x \rho_0(x) dx - \frac{1}{\mu}  \int_{\R^d} \rho_0(x) u_0(x) dx.
\end{equation} 

As for the logarithmic Schrödinger equation, an important feature of \eqref{NMeq} is that the evolution of initial Gaussian data remains Gaussian (see \cite{birula1976} and \cite{carles2018}). In order to shorten the calculations, we may consider centered Gaussian initial data (cf Remark \ref{moving_gaussian_remark} for the behavior of moving-center Gaussian). The following asymptotic result for Gaussian functions will be a crucial guide for the general case, and its proof will be the aim of Section 3. 

\begin{theorem}\textbf{(Gaussian behavior).} \\
Let $\lambda \in \R^*$, and consider the initial data 
\[ \psi_0(t,x) = b_0 \exp \left(-\frac{1}{2} \sum_{j=1}^d a_{0j} x_j^2 \right)  \]
with $b_0$, $a_{0j} \in \C$, $\alpha_{0j} = \Re a_{0j} >0$. Then the solution $\psi$ to \eqref{NMeq} is given by
\[  \psi(t,x)= b_0 \prod_{j=1}^d \frac{1}{\sqrt{r_j(t)}} \exp \left( i \phi_j(t)-\alpha_{0j} \frac{x_j^2}{2r_j^2(t)} + i \frac{\dot{r}_j(t)}{r_j(t)} \frac{x_j^2}{2}  \right)    \]
for some real-valued functions $\phi_j$, $r_j$ depending only on time, such that, as $t \rightarrow \infty$,
\begin{itemize}
  \item[$\bullet$] if $\lambda <0$,
  \[  r_j(t) \rightarrow \sqrt{\frac{ \alpha_{0j}}{-2 \lambda}}  \ \ \ \text{and} \ \ \ \dot{r}_j(t) \rightarrow 0, \]
  so the Gaussian solution $\psi$ tends to a universal Gaussian profile, namely
\[ | \psi(t,x) | \underset{t\to+\infty}{\longrightarrow} C e^{\lambda |x|^2} \ \ \ \text{for all } x \in \R^d,  \]
where $C=C(\lambda,\alpha_0,b_0,d) > 0$ denotes a generic constant depending only on $\lambda$, $\| \psi_0 \|_{L^2(\R^d)}$ and the dimension $d$, 
  \item[$\bullet$] and if $\lambda >0$, for large $t$,
 \[  r(t) \sim 2\sqrt{\frac{\lambda \alpha_0}{\mu}t} \ \ \ \text{and} \ \ \ \dot{r}(t) \sim \sqrt{\frac{\lambda \alpha_0}{\mu t}}, \]
 so that
 \[ \| \psi(t) \|_{L^{\infty}} = \hat{C} t^{-d/4} \ \ \ \text{and} \ \ \ \| \nabla \psi(t) \|_{L^2} \sim  \tilde{C} \frac{1}{\sqrt{t}}.  \]
 where $\hat{C}=\hat{C}(\alpha_0,b_0,\mu,\lambda,d)$ and $\tilde{C}=\tilde{C}(\alpha_0,b_0,\mu,\lambda,d) >0$ denote two generic constants depending only on $\lambda$, $\mu$, the initial conditions and the dimension $d$.
\end{itemize} 
\end{theorem}

We now give the notion of weak solution we will use throughout this paper. Following \cite{antonelli2009}, we express:

\begin{definition} 
We say that $(\rho,J)$ is an \textbf{energy weak solution} of system \eqref{continuity}-\eqref{fluid} in $\left[ 0,T \right[ \times \R^d$ with initial data $(\rho_0,J_0) \in L^1(\R^d) \times  L^1(\R^d)$, if there exists locally integrable functions $\sqrt{\rho}$, $\Lambda$ such that, by defining $\rho := \sqrt{\rho}^2$ and $J=\sqrt{\rho} \Lambda$, the following holds:

\begin{enumerate}[label=(\roman*)]

    \item The global regularity:
    \[ \sqrt{\rho} \in L^{\infty}(\left[ 0,T \right[;H^1(\R^d)), \ \ \ 
      \Lambda  \in L^{2}(\left[ 0,T \right[;L^2(\R^d)),  \]
      with the compatibility condition
      \[ \sqrt{\rho} \geq 0 \text{ a.e. on } (0,\infty) \times \R^d, \ \ \ 
      \Lambda =0 \text{ a.e. on } \left\{ \rho=0 \right\}. \]
      
      \item For any test function $\eta \in \mathcal{C}_0^{\infty}(\left[ 0,T \right[ \times \R^d)$, 
      
      \[ \int_0^T \int_{\R^d} (\rho \partial_t \eta + J \cdot \nabla \eta )dx dt + \int_{\R^d} \rho_0 \eta(0) dx =0, \]
      and for any test function $\zeta \in \mathcal{C}_0^{\infty}(\left[ 0,T \right[ \times \R^d; \R^d)$,
      \begin{gather*} \int_0^T \int_{\R^d} \left( J \cdot \partial_t \zeta + \Lambda \otimes \Lambda : \nabla \zeta +\lambda \rho \Div (\zeta) - \mu J \cdot \zeta + \nabla \sqrt{\rho} \otimes \nabla \sqrt{\rho} : \nabla \zeta - \frac{1}{4} \rho \Delta \Div \zeta \right) dx dt \\ 
      +  \int_{\R^d} J_0 \zeta(0) dx =0. \end{gather*}
      
      \item (Generalized irrotationality condition) For almost every $t \geq 0$, 
      \[ \nabla \wedge J = 2 \nabla \sqrt{\rho} \wedge \Lambda \]
      holds in the sense of distributions.
      
      \item For almost every $t \in \left[0,T\right)$, equation \eqref{eq_energy} holds.

\end{enumerate}
\label{weak}
\end{definition}

In the focusing case, under some regularity and tightness assumption, we will show that every energy weak solution tends weakly in $L^1(\R^d)$ to a universal Gaussian profile of same mass. We denote 
\[ W(\R^d):=\enstq{ \psi \in H^1(\R^d)}{ |\psi|^2 \log |\psi|^2 \in L^1(\R^d)}.  \]

\begin{assumption} \textbf{(Regularity in the focusing case).} \\
Let $\lambda <0$ and $(\rho_0,J_0) \in L^1(\R^d) \times  L^1(\R^d)$. We assume there exists an energy weak solution $(\rho,J)$ of system \eqref{continuity}-\eqref{fluid} with initial data $(\rho_0,J_0)$ in the sense of Definition \ref{weak} with the additional regularity:
\[ \sqrt{\rho} \in L^{\infty}(\left[ 0,T \right[;H^2(\R^d) \cap W(\R^d)). \] \label{assumption_reg_focusing} 
\end{assumption}
\begin{remark}
The condition (iv) in Definition \ref{weak} insures the energy \eqref{energy} to be bounded by above, however the term $\int \rho \log \rho$ has no definite sign, hence the condition $\sqrt{\rho} \in L^{\infty}(\left[ 0,T \right[; W(\R^d))$ ensures that the energy is bounded from below. The regularity assumption on $H^2$ will be used in the following in order to perform the limit in equation \eqref{fluid}.
\end{remark}

\begin{assumption} \textbf{(Tightness in the focusing case).} \\
If $\lambda <0$, we assume that every energy weak solution $(\rho,J)$ of system \eqref{continuity}-\eqref{fluid} in the sense of Definition \ref{weak} has a bounded moment of order 2, namely
\[ \sup_{t\geq 0} \int_{\R^d} |x|^2 \rho(t,x) dx < \infty. \]
\label{assumption_tightness}
\end{assumption}

\begin{theorem}  \textbf{(Focusing case).} \\
If $\lambda <0$, under Assumption \ref{assumption_reg_focusing} and Assumption \ref{assumption_tightness}, and assuming $\rho_0 \neq 0$, we have 
\[ \rho(t,.) \underset{t\to+\infty}{\rightharpoonup} c_{\lambda} e^{\lambda |x-x_{\infty}|^2} \ \ \ \text{weakly in} \ L^1(\R^d),\]
where $ c_{\lambda} :=\| \rho_0 \|_{L^1(\R^d)} (-\lambda/\pi)^{d/2}$ and $x_{\infty}$ is determined by
\begin{equation} \label{x_infty}
   x_{\infty} \| \rho_0 \|_{L^1(\R^d)}  =
   \lim\limits_{\substack{t \to \infty}} \int_{\R^d} x \rho(t,x) dx =\int_{\R^d} x \rho_0(x) dx - \frac{1}{\mu}  \int_{\R^d} \rho_0(x) u_0(x) dx
\end{equation}
using \eqref{center_mass_conv}.
\label{theorem1}
\end{theorem}

\begin{remark}
Assumption \ref{assumption_tightness} deserves a few remarks. First of all, this hypothesis seems to be true for every energy weak solution, as it holds for every Gaussian solutions (see Remark \ref{moving_gaussian_remark}), and seems to be verified through numerical simulations.

Despite probably not being optimal, this hypothesis is crucial in order to show the convergence of $\rho$ to a universal Gaussian profile. In fact, energy estimate \eqref{eq_energy} prevents oscillations and finite-time blow-up of the density $\rho$, however the boundedness of the center of mass $\int x \rho$ still allows two symmetric Gaussian functions to symmetrically diverge with respect to the origin at infinity (although numerics seems to invalidate this kind of scenario). 

Let us explain why the method we use in order to bound $\int x\rho$ fails when it comes to estimate $\int |x|^2 \rho$. As we would like to exploit the crucial property that $E_c \in L^1(\R)$, we try to calculate
\[ \frac{d}{dt} \left( \int_{\R^d}|x|^2 \rho(t,x) dx  \right) = \int_{\R^d} x \rho(t,x) u(t,x) dx :=f(t).  \]
Using equation \eqref{fluid}, we get
\[ f'+\mu f  = \| \nabla \sqrt{\rho} \|^2_{L^2(\R^d)} + \| \sqrt{\rho} u \|^2_{L^2(\R^d)} +\lambda d \| \rho_0\|_{L^1(\R^d)} = 2 E(t)  -2\lambda \int_{\R^d} \rho \log \rho + \lambda d \| \rho_0\|_{L^1(\R^d)},    \]
which implies that $f$ is bounded using the logarithmic Sobolev inequality \eqref{log_sobo_ineq} and the energy estimate \eqref{eq_energy}. Unfortunately, this is insufficient to show that $f$ is integrable.
\end{remark}

For the focusing case ($\lambda >0$), we are going to see that we have an analogous result of what holds for the logarithmic Schrödinger equation. We denote 
\[ \mathcal{F}(H^1)(\R^d):=\enstq{ \psi \in L^2(\R^d)}{ \langle x \rangle \psi(x) \in L^2(\R^d)},  \]
where $\langle x \rangle = \sqrt{1+|x|^2}$.

\begin{assumption} \textbf{(Regularity in the defocusing case).} \\
Let $\lambda >0$ and $(\rho_0,J_0) \in L^1(\R^d) \times  L^1(\R^d)$. We assume there exists an energy weak solution $(\rho,J)$ of system \eqref{continuity}-\eqref{fluid} with initial data $(\rho_0,J_0)$ in the sense of Definition \ref{weak} with the additional regularity:
\[ \sqrt{\rho} \in L^{\infty}(\R;H^1\cap \mathcal{F}(H^1)(\R^d) ). \]
\label{assumption_reg_defocusing}
\end{assumption}

\begin{remark} \label{rem_defocusing}
Note that the above condition ensures that the term $\int \rho \log \rho$ in the energy \eqref{energy} is bounded in the defocusing case (cf \cite{carles2018} or below in the proof of Lemma \ref{lemma_entropy}).
\end{remark}

\begin{theorem} \textbf{(Defocusing case).} \\
If $\lambda >0$, under Assumption \ref{assumption_reg_defocusing}, we denote by $\tilde{\rho}$ the scaled function defined by
\[ \rho(t,x)= \frac{1}{\tau(t)^{d/2} }\tilde{\rho}\left( t,\frac{x}{\tau(t)} \right), \]
where  $\tau $ denotes the unique $\mathcal{C}^{\infty}(\left[ 0,\infty \right))$ solution of the differential equation
\begin{equation*}
\ddot{\tau} = \frac{ 2 \lambda}{\tau} - \mu \dot{\tau}, \ \ \ \tau(0)=1, \ \ \ \dot{\tau}(0)=0,
\end{equation*}
which satisfies, as $t \rightarrow +\infty$,
\[ \tau(t) \sim 2 \sqrt{\frac{\lambda t}{\mu}} .    \]
Then 
\[ \tilde{\rho}(t,.) \underset{t\to+\infty}{\rightharpoonup} c_0 e^{-|x|^2} \ \ \ \text{weakly in} \ L^1(\R^d),\]
where $c_0 :=\| \rho_0 \|_{L^1(\R^d)} \pi^{-d/2}$.
\label{theorem2}
\end{theorem}

%%\begin{remark}
%%In particular, in the focusing case, every energy weak solution tends to 0 in $L^{\infty}$ norm.
%%\end{remark}

\begin{remark} \textit{(Effect of scaling factors)} \\
We conclude this section with a last property that equation \eqref{NMeq} shares with the logarithmic Schrödinger equation \eqref{logNLS}. Unlike the typical power-like nonlinear Schrödinger equation \eqref{powerNLS}, replacing $\psi$ with $\kappa \psi$ ($\kappa >0$) in \eqref{NMeq} has only little effect. In fact, if $\psi$ is solution of \eqref{NMeq} with initial datum $\psi(0,x)=\psi_0$, then the purely time-dependent gauge transform
\[ \kappa \psi \exp \left(i \frac{\lambda}{\mu} \log |\kappa|^2 \left(1- e^{-\mu t} \right) \right) \]
is also a solution of \eqref{NMeq} with initial datum $\psi(0,x)=\kappa \psi_0$. In particular, the $L^2$-norm of the initial datum has no influence on the dynamics of the solution.

\end{remark}

\section{Gaussian solutions}

\subsection{Propagation of Gaussian data}

The following calculations will be made in dimension $d=1$ for reader's convenience, but they can all be adapted in any dimension $d$ (cf Remark \ref{remark_d} at the end of this section).\\
We plug into \eqref{NMeq} Gaussian functions of the form:
\begin{equation}
 \psi(t,x) = b(t) \exp\left(-\frac{1}{2} a(t) x^2 \right), \label{expression_psi}
 \end{equation}
so we have the equation
\[ i \dot{b} -\frac{1}{2} i \dot{a} b x^2 - \frac{1}{2} ab + \frac{1}{2} a^2 b x^2 = \lambda b \log(|b|^2) - \lambda b \Re(a) x^2 - \frac{1}{2}i \mu b \log \left( \frac{b}{b^*} \right) - \frac{1}{2} \mu b \Im(a) x^2. \]

Equating the constant in $x$ and the factors of $x^2$, we get the following non-linear differential equations system:
 \begin{align}
    i \dot{b} - \frac{1}{2}ab  = \lambda b \log(|b|^2)- \frac{1}{2} i\mu b \log\left(\frac{b}{b^*} \right),\label{eqb1}  \\
    i \dot{a}- a^2  = 2 \lambda \Re(a) + \mu \Im(a).  \label{eqa1}
\end{align}   

We first look at equation \eqref{eqb1} in order to express $|b|$ as a function of $a$. Multiplying \eqref{eqb1} by $b^*$, we get 
\[ i \dot{b} b^* - \frac{1}{2}a|b|^2  = \lambda |b|^2 \log(|b|^2)- \frac{1}{2} i\mu |b|^2 \label{eqb} \log\left(\frac{b}{b^*} \right). \]

As $i\log(b/b*) \in \R$, by taking the imaginary part, we obtain
\[ \Im \left( i \dot{b} b^* - a \frac{|b|^2}{2} \right)= 0.    \]
We recall that  $\Re ( \dot{b} b^* ) = \frac{1}{2} \frac{d}{dt}(|b|^2)$, so we have the differential equation 
\[  \frac{d}{dt}(|b|^2) = |b|^2 \Im(a),    \]
so
\begin{equation}
|b(t)|=|b(0)| \exp \left( \frac{1}{2} \Im A(t) \right) \label{expression_b}
\end{equation}
with $ A(t)=\int_0^t a(s) ds $. 

\begin{remark} 
We can also express $b$ as a function of $a$. We then look at $b$ under the form
\[ b(t)=r(t) e^{i \phi(t)},  \]
where $r(t)=|b(t)|$. By plugging this expression into \eqref{eqb1}, we obtain the equation
\[ i \dot{r} - \dot{\phi} r - \frac{1}{2} ar = \lambda r \log(r^2) +\mu \phi r .      \]
As $\dot{r}/r=\frac{1}{2} \Im(a(t)) $, we have
\[ \dot{\phi} + \mu \phi = \frac{1}{2} i \Im(a) - \frac{1}{2}a - \lambda \log(r^2) = \frac{1}{2} i \Im(a) - \frac{1}{2}a - \lambda \log(|b_0|^2) - \lambda \Im(A). \]
Thanks to the variation of constant formula, we get $\phi$ as the solution of this ordinary differential equation:
\[ \phi(t)=\phi_0 e^{\mu t} + \int_0^t e^{\mu(t-s)} \left( \frac{1}{2} i \Im(a(s)) - \frac{1}{2} a(s) - \lambda  \log(|b_0|^2) - \lambda \Im(A(s)) \right) ds,   \]
and so we get $b$ as a function of $a$.
\end{remark}

We now want to solve equation \eqref {eqa1} with initial conditions $a(0)=\alpha_0 + i \beta_0$, with $\alpha_0 >0$. As suggested in \cite{carles2018}, we denote
\[ a= -i \frac{ \dot{\omega}}{\omega}, \]
so we have
\[ i \dot{a} - a^2 = \frac{ \ddot{\omega} \omega - \dot{\omega}^2}{ \omega^2} + \left( \frac{\dot{\omega}}{\omega} \right)^2 = 2 \lambda \Re\left(-i \frac{\dot{\omega}}{\omega} \right) + \mu \Im\left(-i \frac{\dot{\omega}}{\omega} \right), \]
that gives the equation
\[ \ddot{\omega} = 2 \lambda \omega \Im \left(\frac{\dot{\omega}}{\omega} \right) - \mu \omega \Re \left(\frac{\dot{\omega}}{\omega} \right). \] 

Denoting $\omega = r e^{i \theta}$, we calculate
\begin{align*}
    \dot{\omega} & =  \dot{r} e^{i \theta} +ir \dot{\theta} e^{i \theta}, \\  
    \ddot{\omega} & =  \ddot{r} e^{i \theta} + 2i \dot{r} \dot{\theta} e^{i \theta} - r \dot{\theta}^2e^{i \theta}+ir \ddot{\theta} e^{i \theta}, \\
    2 \lambda \omega \Im \left(\frac{\dot{\omega}}{\omega} \right) & = 2 \lambda r e^{i \theta} \Im \left( \frac{\cancel{\dot{r}}e^{i \theta} +i r \dot{\theta} e^{i \theta}}{r e^{i \theta}} \right) = 2 \lambda r \dot{\theta} e^{i \theta},  \\   
     \mu \omega \Re \left(\frac{\dot{\omega}}{\omega} \right) & =  \mu r e^{i \theta} \Re \left( \frac{\dot{r}e^{i \theta} +\cancel{i r \dot{\theta}} e^{i \theta}}{r e^{i \theta}} \right) =  \mu \dot{r}e^{i \theta}.    
\end{align*}

Then we get the equation
\[ \ddot{r} +2ir \dot{\theta} - r \dot{\theta}^2 + i r \ddot{\theta} =2 \lambda r \dot{\theta} - \mu \dot{r}. \]

By taking the real and imaginary part, we get the real system of equations:
\begin{align}
    \ddot{r} - r \dot{\theta}^2 = 2 \lambda r \dot{\theta} - \mu \dot{r}  \label{eqrtheta1} \\
    r \ddot{\theta} + 2 \dot{r} \dot{\theta} = 0. \label{eqtheta1}
\end{align}

We know that
\[ \dot{\theta}(0)=\alpha_0 \ \ \ \text{and} \ \ \ \left( \frac{\dot{r}}{r} \right)(0)= - \beta_0 .\]

We can fix $r(0)=1$ so that $\dot{r}(0)=-\beta_0$. We note that \eqref{eqtheta1} gives us the conservation of the quantity
\[ \frac{d}{dt} (r^2\dot{\theta}) = r(2 \dot{r}\dot{\theta} + r \ddot{\theta}) =0, \]
so $\dot{\theta} = \alpha_0/r^2$. By plugging $\dot{\theta}$ into \eqref{eqrtheta1}, we finally obtain the following equation on $r$:
\begin{align} \boxed{ \ddot{r} = \frac{ \alpha_0^2}{r^3} + \frac{ 2 \lambda \alpha_0}{r} - \mu \dot{r}, } \label{eqr} \end{align}
with $r(0)=1$ and $\dot{r}(0)=-\beta_0$. By multiplying \eqref{eqr} by $\dot{r}$ and integrating between 0 and $t$, we get:
\begin{equation}
\dot{r}(t)^2 = \dot{r}(0)^2 + \alpha_0^2 \left( 1 - \frac{1}{r(t)^2} \right) + 4 \lambda \alpha_0 \log(r(t)) - 2 \mu \int_0^t\dot{r}(s)^2 ds. 
\label{eq_int}
\end{equation}

Recalling the expression of $a$, we get
\begin{equation}
a(t)= \frac{\alpha_0}{r(t)^2} -i \frac{\dot{r}(t)}{r(t)}, \label{expression_a}
\end{equation} 
in particular $\Re(a) \geq 0$. We note that we have an explicit expression of our Gaussian $\psi$ as function of $r$. As we are facing a non-linear equation of order 2, we will now look at the asymptotic behavior of its solutions for both the defocusing and the focusing case. This will give the asymptotic dynamics of our Gaussian solutions.

\subsection{Focusing case}

We first look at the focusing case which corresponds to $\lambda <0$ in \eqref{eqr}. The Cauchy theory ensures the existence of a $\mathcal{C}^2$-solution (denoted by $r$ in the following) of \eqref{eqr} as long as $r \neq 0$. We will show in this section that every Gaussian tends a multiple of the Gaussian limit $e^{\lambda x^2}$.

\begin{lemma}
There exist $m$, $M>0$ such that for all $t\geq 0$, 
\[ 0 <m \leq r(t) \leq M. \]
\end{lemma}
\begin{proof}
Using equation \eqref{eq_int}, as 
\[ \dot{r}(t)^2 +  \frac{\alpha_0^2}{r(t)^2} + 2 \mu \int_0^t\dot{r}(s)^2 ds \geq 0,   \]
we get that
\[ \dot{r}(0)^2 + \alpha_0^2  + 4 \lambda \alpha_0 \log(r(t)) \geq 0.  \]

Then 
\[ \log(r(t)) \leq \frac{\dot{r}(0)^2 + \alpha_0^2 }{- 4\lambda \alpha_0} \]
as $\lambda <0$, so $r(t) \leq M:= \exp(-(\dot{r}(0)^2 + \alpha_0^2 )/4\lambda \alpha_0)$. In order to get a lower bound, we remark that equation \eqref{eq_int} also gives
\[ \dot{r}(0)^2 + \alpha_0^2  + 4 \lambda \alpha_0 \log(r(t)) \geq 2 \mu \int_0^t\dot{r}(s)^2 ds .   \]

\underline{Case 1:} $\dot{r}(0) \neq 0$.\\
Then for $t_0 >0$ fixed (that could be taken as small as we want), continuity of $\dot{r}$ implies that there exists $ c_{t_0} >0$ such that $2 \mu \int_0^t\dot{r}(s)^2 ds \geq c_{t_0}$ for all $t \geq t_0$. Using again  equation \eqref{eq_int}, we have that 
\[ \dot{r}(0)^2 + \alpha_0^2  + 4 \lambda \alpha_0 \log(r(t)) \geq \frac{\alpha_0^2}{r(t)^2},  \]
so, for all $t\geq t_0$,
\[ r(t) \geq \frac{\alpha_0^2}{\sqrt{\dot{r}(0)^2 + \alpha_0^2  + 4 \lambda \alpha_0 \log(r(t)) } } \geq \frac{\alpha_0}{\sqrt{c_{t_0}}}. \]
We can now choose $t_0$ such that for all $t \leq t_0$, $r(t) \geq 1/2$ (as we know that $r$ is continuous and $r(0)=1$), and the result holds by taking $m=\min(1/2,\alpha_0/\sqrt{c_{t_0}})$.\\

\underline{Case 2:} $\dot{r}(0) = 0$.\\
If for all $t \geq 0$, $\dot{r}(t)=0$, then $r=1$ and the result is obvious. If it exists $T>0$ such that $\dot{r}(T) \neq 0$, then a time shift $t \mapsto t-T$ in equation \eqref{eqr} brings us back to the Case 1. Note that this case will be the one used as the generic case in the following sections.

%% Now, integrating equation \eqref{eqr}, we have
%%\[ \dot{r}(t) = \dot{r}(0) + \int_0^t \left(  \frac{ \alpha_0^2}{r^3(s)} + \frac{ 2 \lambda \alpha_0}{r(s)} \right) ds - \mu (r(t) - r(0)). \]
%% Assume by contradiction that for all $\eps >0$, there exists a time $t_{\eps} \geq 0$ such that $r(t_{\eps}) \leq \eps$. Let us remark that we can take $t_{\eps}$ minimal for this property in order to insure $\dot{r}(t_{\eps}) \leq 0$. Thanks to the previous formula we get the lower bound :
%% \[ \dot{r}(t_{\eps}) \geq \dot{r}(0) + \mu r(0) +  \left(  \frac{ \alpha_0^2}{\eps^3} + \frac{ 2 \lambda \alpha_0}{\eps} \right) t_{\eps} - \mu \eps.  \]
%% As $\eps^3 \ll \eps$ while $\eps \rightarrow 0$, we can have $\dot{r}(t_{\eps}) >0$ for $\eps $ small enough looking at the previous equation, which is impossible. Then there exists $m>0$ such that for all $t\geq 0$, $r(t) \geq m$.
\end{proof}

Before going for the proof of Proposition \ref{prop_r_focusing}, we recall a lemma from real analysis:

\begin{lemma} \label{lemma_unif_cont}
Let $f:\R_+ \rightarrow \R_+$ be a nonnegative uniformly continuous function, and assume that $ \int_0^{\infty} f(x) dx  < \infty$. Then $f(x) \rightarrow 0$ as $x \rightarrow +\infty$.
\end{lemma}

\begin{proposition} \label{prop_r_focusing}
We denote by $r$ the $\mathcal{C}^2$-solution on $\left[ 0, +\infty\right[$ of equation \eqref{eqr} with $\lambda<0$. Then, as $t\rightarrow + \infty$, we have 
\[  r(t) \rightarrow \sqrt{\frac{ \alpha_0}{-2 \lambda}}  \ \ \ \text{and} \ \ \ \dot{r}(t) \rightarrow 0. \]
\end{proposition}
\begin{proof}
Still using equation \eqref{eq_int}, as we can remark that
\[  \frac{\alpha_0^2}{r(t)^2} + 2 \mu \int_0^t\dot{r}(s)^2 ds \geq 0,   \]
we get that
\[ \dot{r}(t)^2 \leq \dot{r}(0)^2 + \alpha_0^2  + 4 \lambda \alpha_0 \log(r(t)), \]
and as we know that $m \leq r(t) \leq M$ for all $t \geq 0$, we also get that $\dot{r}$ is bounded, as well as $\ddot{r}$ using equation \eqref{eqr}. So $\dot{r} \ddot{r}$ is also bounded on $\R_+$, hence $\dot{r}^2$ is Lipschitz continuous so uniformly continuous. We assume by contradiction that $\dot{r}(t) \cancel{\rightarrow} 0$ at $+\infty$, so according to Lemma \ref{lemma_unif_cont} we have 
\[ \int_0^t \dot{r}(s)^2 ds \underset{t\to+\infty}{\longrightarrow} +\infty. \]
 By the integral expression \eqref{eq_int} and as $r$ is bounded and $\mu >0$ we get that $\dot{r}(t)^2 \rightarrow -\infty$ as $t \rightarrow \infty$, which is obviously impossible. So we get that
\[ \dot{r}(t) = \dot{r}(0) +\int_0^t \ddot{r}(s) ds \underset{t\to+\infty}{\longrightarrow}0,   \]
so $ \int_0^{\infty} \ddot{r}(s) ds $ has a finite limit (in particular, it is uniformly bounded). Multiplying equation \eqref{eqr} by $\ddot{r}$, integrating over $\left[ 0 ,t \right]$ and using the lower bound of $r$, we get that
\[ \int_0^t  \ddot{r}(s)^2 ds \leq \left( \frac{\alpha_0^2}{m^3} + \frac{2 \lambda \alpha_0}{m}  \right) \int_0^t  \ddot{r}(s) ds -\frac{\mu}{2} (\dot{r}(t)^2- \dot{r}(0)^2),  \]
so $ \int_0^{\infty} \ddot{r}(s)^2 ds  < \infty$. Differentiating equation \eqref{eqr}, we have
\[ \dddot{r} = -\frac{ 3\alpha_0^2 \dot{r}}{r^4} - \frac{ 2 \lambda \alpha_0  \dot{r}}{r^2} - \mu \ddot{r},  \]
so $\dddot{r}$ is bounded on $\R_+$, and $\ddot{r}^2$ is Lipschitz so uniformly continuous. Then by applying Lemma \ref{lemma_unif_cont} we have $\ddot{r}(t)^2 \rightarrow 0$ as $t \rightarrow \infty$, so $\ddot{r}\rightarrow 0$.\\
We take $(t_n)_n$ a sequence such that $t_n \rightarrow + \infty$ as $n \rightarrow \infty$. As $r$ is bounded, there is a subsequence (not relabeled here for convenience) such that $r(t_n) \rightarrow L$. We already know that $\dot{r}(t_n) \rightarrow 0$ and $\ddot{r}(t_n) \rightarrow 0$, so by passing in the limit in equation \eqref{eqr}, we get that
\[ \frac{\alpha_0^2}{L^3} + \frac{2 \lambda \alpha_0}{L} =0,  \]
and we deduce that $L= \sqrt{-\alpha_0/2 \lambda}$. As the sequence $(r(t_n))_n$ has a unique point of accumulation, by sequential characterisation of continuity we get that  
\[  r(t) \rightarrow \sqrt{\frac{ \alpha_0}{-2 \lambda}}.   \]
\end{proof}

Hence 
\[ a(t)= \frac{\alpha_0}{r(t)^2} -i \frac{\dot{r}(t)}{r(t)} \underset{t\to+\infty}{\longrightarrow} -2\lambda, \]
so for all $x \in \R$ we get that
\[ | \psi(t,x) | \underset{t\to+\infty}{\longrightarrow} C e^{\lambda x^2},  \]
where $C=C(\lambda,\alpha_0,\beta_0) > 0$ is a constant depending only on the initial conditions and $\lambda$. 
\begin{remark}
As we have $\| \psi(t) \|_{L^2(\R)} = \| \psi_0 \|_{L^2(\R)}$ for all $t \in \R_+$, we can make explicit the constant $C(\lambda,\alpha_0,\beta_0) = \sqrt[4]{-2\lambda/\pi} \| \psi_0 \|_{L^2(\R)} $.
\end{remark}

%%\begin{remark}
%%Using equation \eqref{eqr}, we get that the rate of convergence to the Gaussian limit is of order $ \dot{r}(t) \sim e^{- \mu t}$.
%%\end{remark}

\begin{remark} \label{moving_gaussian_remark}
In order to check the tightness hypothesis in the Gaussian case, we can make the same calculations for Gaussian functions with moving centers as performed in \cite{carles_rigidity}, using the hydrodynamical equations of motion \eqref{continuity}-\eqref{fluid} this time, where $\rho$ denotes the density and $u$ the velocity of our Gaussian functions. Taking different initial centers leads to considering 
\begin{equation*} \label{moving_gaussian_CI}
    \rho(0,x)= b_0 e^{-\alpha_0 x^2}, \ \ \ u(0,x)=\beta_0 x + c_0,
\end{equation*}
for all $x \in \R$, with $b_0$, $\alpha_0 >0$ and $\beta_0$, $c_0 \in \R$. We then look at solutions of the form
\begin{equation} \label{moving_gaussian}
    \rho(t,x)= b(t) e^{-\alpha(t) (x-\overline{x}(t))^2}, \ \ \ u(t,x)=\beta(t) x + c(t),
\end{equation}
and plug this ansatz in equations \eqref{continuity}-\eqref{fluid}, leading to the system of ordinary differential equations
\begin{equation} \label{eqdiff_moving1}
    \dot{\alpha}+2\alpha \beta =0, \ \ \ \dot{\beta} + \beta^2 + \mu \beta = 2 \lambda \alpha + \alpha^2,
\end{equation}
\begin{equation} \label{eqdiff_moving2}
    \dot{\overline{x}} = \beta \overline{x} + c, \ \ \ \dot{c} + \beta c + \mu c = -2 \lambda \alpha \overline{x} - \alpha^2\overline{x},
\end{equation}
\begin{equation} \label{eqdiff_moving3}
   \dot{b} = b (\dot{\alpha} \overline{x}^2 + 2 \alpha \overline{x} (\dot{\overline{x}} - c) - \beta ) . 
\end{equation}
In order to solve this system, mimicking \cite{li_wang_2006} and \cite{carles_rigidity}, we can check that the two equations of \eqref{eqdiff_moving1} are satisfied if and only if $\alpha$ and $\beta$ are of the form
\[ \alpha(t)= \frac{\alpha_0}{r(t)^2}, \ \ \ \beta(t) = \frac{\dot{r}(t)}{r(t)}, \]
where we recall that $r$ is the solution of equation \eqref{eqr} with $r(0)=1$ and $\dot{r}(0)= \beta_0$. Plugging these expressions into \eqref{eqdiff_moving3}, we also get that
\[ b(t) = \frac{b_0}{r(t)}. \]
We know that $\rho$ is a Gaussian function centered in $\overline{x}$, so we get that for all $t \geq 0$, 
\[ \int_{\R} (x - \overline{x}(t)) \rho(t,x) dx =0, \]
hence $\overline{x}(t) = \int_{\R} x \rho(t,x) dx / \| \rho_0 \|_{L^1(\R)}$, and we get from \eqref{center_mass_conv} that $\overline{x}(t)$ has a finite limit as $t \rightarrow +\infty$. In particular, we note that the dissipation implies that the center of our Gaussian solutions is bounded, contrary to the logarithmic case ($\mu = 0$). This stands as an evidence that the Galilean invariance principle is no longer verified due to the dissipation term ($\mu >0$). In order to get the behavior of the second moment of $\rho$, we calculate:
\begin{align*}
    \int_{\R} x^2 \rho(t,x) dx  & =  \int_{\R} (x - \overline{x}(t) +\overline{x}(t))^2 \rho(t,x) dx \\
       & =  \int_{\R} (x - \overline{x}(t) )^2 \rho(t,x) dx + 2 \overline{x}(t) \int_{\R} (x - \overline{x}(t) ) \rho(t,x) dx  +  \int_{\R} \overline{x}(t)^2 \rho(t,x) dx \\
       & = \frac{b(t)}{2 \alpha(t)} +  \overline{x}(t)^2 \int_{\R}  \rho_0(x) dx = \frac{b_0}{2 \alpha_0} r(t)  + \overline{x}(t)^2 \| \rho_0 \|_{L^1(\R)} . \\
\end{align*}     
As we know that $r$ and $\overline{x}$ converges thanks to the previous computations, we get that $\int_{\R} x^2 \rho$ is uniformly bounded in time, which corroborate Assumption \ref{assumption_tightness}.

%%Equations  \eqref{eqdiff_moving2} then give
%%\begin{subequations} 
%%\begin{empheq}[left=\empheqlbrace]{align}
%% & \dot{\overline{x}} = \frac{\dot{r}}{r} \overline{x} + c, \label{eq_center1} \\
%% & \dot{c} + \frac{\dot{r}}{r}c + \mu c = -\left(\frac{\dot{r}}{r} + \mu \frac{\dot{r}}{r} \right) \overline{x}. \label{eq_center2}
%%\end{empheq} 
%%\end{subequations}
%%Integrating equation \eqref{eq_center1}, we get that 
%%\begin{equation} 
%%\overline{x}(t) = \frac{1}{r} \left( x_0 + \int_0^t r(s) c(s) ds \right),  \label{exp_center}  
%%\end{equation}
%%so in order to show that $\overline{x}$ converges, we need to show that $y:= x_0 + \int_0^t r(s) c(s) ds$ has a finite limit. Plugging \eqref{exp_center} into equation \eqref{eq_center2}, we get the following homogeneous linear differential equation on $y$ :
%%\begin{equation} \label{eq3}
%%    y'' + \mu y' + (1+\mu)\frac{\dot{r}}{r} y = 0.
%%\end{equation}

\end{remark}

\subsection{Defocusing case}

We now look at the defocusing case by taking $\lambda <0$ in \eqref{eqr}. By calculating an equivalent of $r(t)$ as $t \rightarrow \infty$, we will see that every Gaussian vanishes to 0 in $L^{\infty}$-norm in $t^{-\frac{1}{4}}$ (we recall that $d=1$ for the present computations). We will also make explicit the decay of $\| \nabla \psi(t) \|_{L^2}$ thanks to an equivalent of $\dot{r}$.

\begin{lemma}
There exists a constant $m >0$ such that $\forall t \geq 0$, $r(t) \geq m$.
\end{lemma}
\begin{proof}
Let $t \geq 0$, and we rewrite the expression \eqref{eq_int} under the form:
\[  \dot{r}(t)^2 + \frac{\alpha_0^2}{r(t)^2}  + 2 \mu \int_0^t\dot{r}(s)^2 ds= \dot{r}(0)^2 + \alpha_0^2 + 4 \lambda \alpha_0 \log(r(t)) \geq 0,\] 
so
\[ \log(r(t)) \geq - \frac{\dot{r}(0)^2 + \alpha_0^2}{4 \lambda \alpha_0 }. \]
Denoting $m = \exp \left(  - \frac{\dot{r}(0)^2 + \alpha_0^2}{4 \lambda \alpha_0 }  \right)$, we immediately get that $\forall t \geq 0$, $r(t) \geq m$.
\end{proof}

\begin{lemma}
There exists $T_0 \geq 0$ such that for all $t > T_0$, $\dot{r}(t) > 0$. \label{dotrlemma}
\end{lemma}

\begin{proof}
We first show the existence of a time $T_0$ such that $\dot{r}(T_0) \geq 0$. If $\dot{r}(0) \geq 0$, the result is trivial. Otherwise, $\dot{r}(0) < 0 $, and we denote 
\[ T_0 = \inf \enstq{t \geq 0 }{ \dot{r}(t) \geq 0} \leq +\infty. \]
So $\forall t \in \left[ 0, T_0 \right]$, we have
\[ \ddot{r}(t) = \frac{ \alpha_0^2}{r^3(t)} + \frac{ 2 \lambda \alpha_0}{r(t)} - \mu \dot{r}(t) \geq \frac{ \alpha_0^2}{r^3(t)} + \frac{ 2 \lambda \alpha_0}{r(t)}. \]
As $\dot{r} < 0$ on $\left[ 0, T_0 \right[$, $r$ is decreasing on this interval, so $\forall t \in \left[ 0, T_0 \right[$,
\[  \ddot{r}(t) \geq \frac{ \alpha_0^2}{r^3(0)} + \frac{ 2 \lambda \alpha_0}{r(0)} > 0.   \]
By integration,
\[  \dot{r}(t) \geq \dot{r}(0) + \left( \frac{ \alpha_0^2}{r^3(0)} + \frac{ 2 \lambda \alpha_0}{r(0)} \right) t .  \]
By linear growth, we deduce that $T_0$ is finite, and that $\dot{r}(T_0)=0$ by continuity. We remark that we have the bound
\[ T_0 \leq  - \frac{ \dot{r}(0)}{ \alpha_0^2+ 2 \lambda \alpha_0},   \]
still in the case where $\dot{r}(0) < 0$ (otherwise $T_0=0$). \\
We are going to show now that $\forall t > T_0$, $\dot{r}(t) > 0$. We denote 
\[ T = \inf \enstq{t \geq T_0 }{ \dot{r}(t) \leq 0} \leq +\infty. \]
If $T=+\infty$, we immediately get the result. Otherwise, by a continuity argument, $\dot{r}(T)=0$ and
\[ \ddot{r}(T) = \frac{ \alpha_0^2}{r(T)^3} + \frac{ 2 \lambda \alpha_0}{r(T)} > 0, \]
so for $\eta$ small enough, $\ddot{r}(T-\eta) >0$, and for $\eps > 0$ small enough, $\forall t \in \left] T-\eta, T -\eta+ \eps \right[$, we have $\ddot{r}(t) > 0$, so $\dot{r}$ is increasing on this interval, and in particular $\dot{r}(t) > 0$. Spreading this argument, we see that $\forall t > T_0$, $\dot{r}(t) >0$.
\end{proof}

\begin{lemma}
$r(t) \rightarrow +\infty$ as $t \rightarrow +\infty$.
\end{lemma}
\begin{proof}
According to the previous lemma, we know that $t$ is non-decreasing from a time $T_0$ on, so $r$ is either diverging to $+\infty$, either converging. We assume by contradiction that there exists $\ell > 0$ (as $r \geq m >0$) such that $r(t) \rightarrow \ell$ when $t \rightarrow +\infty$.\\
We assume, still by contradiction, that $\dot{r}(t) \cancel{\rightarrow} 0$ at $+\infty$. As $\dot{r}^2$ is positive and uniformly continuous, we have
\[ \int_0^t \dot{r}(s)^2 ds \rightarrow +\infty, \]
and as $r$ converges by hypothesis, we get from \eqref{eq_int} that $r(t)^2 \rightarrow -\infty$ when $t \rightarrow +\infty$, which is obviously impossible, so $\dot{r}(t) \rightarrow 0$. \\
Then we deduce that
\[ \ddot{r}(t) \rightarrow \frac{\alpha_0^2}{\ell^3}+\frac{2 \lambda \alpha_0}{\ell} >0, \]
so $\exists T \geq 0$ such that $\forall t \geq T$, $\ddot{r}(t) > 0$, $ie$ $\dot{r}$ is increasing on $\left[T,+\infty \right[$. However we saw that $\dot{r}(t) \rightarrow 0$, so $\forall t \geq T$, $\dot{r}(t) > 0$, which is in contradiction with the previous lemma. We finally conclude that $r(t) \rightarrow +\infty$.
\end{proof}

\begin{lemma}
There exists $T_1 \geq 0$ such that for all $t \geq T_1$, $\ddot{r}(t) \leq 0$. \label{ddotrlemma}
\end{lemma}
\begin{proof}
\underline{Case 1:} $\ddot{r}(T_0) \leq 0$. \\
Then we can take $T_1 = T_0$. We assume that there exists $T \geq T_1$ such that $\ddot{r}(T)=0$, then
\[ \dddot{r}(T)=- \frac{3 \alpha_0^2 \dot{r}(T)}{r(T)^4} - \frac{2 \lambda \alpha_0\dot{r}(T)}{r(t)^2} - \mu \ddot{r}(T)  =- \left( \frac{3 \alpha_0^2 }{r(T)^4} - \frac{2 \lambda \alpha_0}{r(T)^2} \right) \dot{r}(T), \]
but $T_1 \geq T_0$ so $\dot{r}(T) > 0$, and so $\dddot{r}(T) < 0$. We deduce that there exists $\eps >0$ small enough such that $\forall t \in \left] T,T+ \eps \right[$, $\ddot{r}(t)<0$. We have just shown that the set $\enstq{t \geq T_1}{\ddot{r}(t) \leq 0}$ is open, and it is clearly closed as $\ddot{r}$ is continuous, and non-empty because $T_0$ belongs to it, so we see that $\forall t \geq T_1$, $\ddot{r}(t) \leq 0$ as $\left[ T_1, +\infty \right[$ is a connected space. \\

\underline{Case 2:} $\ddot{r}(T_0) > 0$. \\
Rewriting this condition with equation \eqref{eqr}, we have:
\[ \mu \dot{r}(T_0) < \frac{\alpha_0^2}{r(T_0)^3} + \frac{2 \lambda \alpha_0}{r(T_0)}. \]
As taking $T_0=T_0 + \eta $ with $\eta$ small enough, we can assume that $\dot{r}(T_0) >0$ and that $\forall t \geq T_0$, $\dot{r}(t) >0$. First of all, we show that $\dot{r}$ is bounded. We denote 
\[ t_{\min}= \inf \enstq {t \geq T_0}{\mu \dot{r}(t)=\frac{\alpha_0^2}{r(T_0)^3} + \frac{2 \lambda \alpha_0}{r(T_0)}}.  \]
We assume (by contradiction) that $t_{\min} < \infty$. We know that $\forall t \geq T_0$, $\dot{r}(t) >0$, so $r$ is increasing on $\left[ T_0,+\infty \right[$. We deduce that
\[ \mu \dot{r}(t_{\min})= \frac{\alpha_0^2}{r(T_0)^3} + \frac{2 \lambda \alpha_0}{r(T_0)}  > \frac{\alpha_0^2}{r(t_{\min})^3} + \frac{2 \lambda \alpha_0}{r(t_{\min})},  \]
so that $\ddot{r}(t_{\min}) < 0$. However $t_{\min}$ is the smallest time $t$ such that $\mu \dot{r}(t) = \alpha_0^2/r(T_0)^3+2\lambda \alpha_0/r(T_0)$, then for all $t < t_{\min}$, $\dot{r}(t) < \dot{r}(t_{\min})$, and
\[ \frac{\dot{r}(t_{\min}) - \dot{r}(t)}{t_{\min}-t} >0, \]
so we have $\ddot{r}(t_{\min}) \geq 0$ when $t\rightarrow t_{\min}$, hence the contradiction. We can conclude that $t_{\min}=+\infty$, so $\forall t \geq T_0$, 
\[ \dot{r}(t) \leq \frac{1}{\mu} \left( \frac{\alpha_0^2}{r(T_0)^3} + \frac{2 \lambda \alpha_0}{r(T_0)} \right), \]
which means that $\dot{r}$ is bounded. \\
We now assume, still by contradiction, that $\forall t \geq T_0$, $\ddot{r}(t) >0$, so $\dot{r}$ is increasing, and since $\dot{r}$ is bounded, it converges to a real $\ell > 0$ (because $\dot{r}(T_0)>0$). As $r(t) \rightarrow +\infty$, by \eqref{eqr}, we have
\[ \ddot{r}(t) \underset{t\to+\infty}{\longrightarrow} -\mu \ell < 0, \]
leading to an absurdity. So there exists a time $T_1 \geq T_0$ such that $\ddot{r}(T_1) \leq 0$. Then we conclude like the Case 1.

\end{proof}

\begin{proposition} \label{prop_r_defocusing}
As  $t \rightarrow + \infty$, we have
\[ r(t) \sim 2\sqrt{\frac{\lambda \alpha_0}{\mu}t}.  \]
\end{proposition}

\begin{proof}
We know from the previous lemma that $\forall t \geq T_1$, $\ddot{r}(t) \leq 0$, so 
\[ \frac{ \alpha_0^2}{r^3(t)} + \frac{ 2 \lambda \alpha_0}{r(t)} \leq \mu \dot{r}(t), \]
and $ \alpha_0^2/r^3(t)\geq 0$, so $\mu \dot{r}(t) \geq \frac{ 2 \lambda \alpha_0}{r(t)} $. We denote by $g$ the solution of the differential equation
\begin{equation} \mu \dot{g} = \frac{ 2 \lambda \alpha_0}{g} \label{eqdiffy}, \end{equation}
with $g(0)=r(0)$, so that $\forall t \geq T_1$, $r(t) \geq g(t)$. Solving \eqref{eqdiffy}, we get the lower bound:
\begin{equation} \label{mino} r(t) \geq g(t) =\sqrt{ \frac{ 4 \lambda \alpha_0}{\mu} t +g(0)^2} \geq  2\sqrt{\frac{\lambda \alpha_0}{\mu}t}. \end{equation}

We now need an upper bound for $r$. We rewrite the equation \eqref{eqr} under the form $\mu \dot{r} = \alpha_0^2/r^3 +  2 \lambda \alpha_0/r - \ddot{r}$, then by integrating between the times $T_1$ and $t$, we have
\[ \mu r(t) = \mu r(T_1) + \int_{T_1}^t \left( \frac{ \alpha_0^2}{r^3(s)} + \frac{ 2 \lambda \alpha_0}{r(s)}  \right) ds  - \dot{r}(t) + \dot{r}(T_1) \leq  \mu r(T_1) + \dot{r}(T_1) + \int_{T_1}^t \left( \frac{ \alpha_0^2}{r^3(s)} + \frac{ 2 \lambda \alpha_0}{r(s)}  \right) ds , \]
because $\dot{r}(t) \geq 0$. Using the previous lower bound $r$ in order to bound $1/r$ by above, we get
\[ r(t) \leq \frac{r(T_1) + \dot{r}(T_1)}{\mu} + \int_{T_1}^t  \frac{1}{8 \lambda} \sqrt{\frac{ \alpha_0 \mu }{ \lambda s^3}}   ds +\int_{T_1}^t  2\sqrt{\frac{\lambda \alpha_0}{\mu s}}   ds. \]

Recalling that the function $s \mapsto -\frac{3}{2\sqrt{s}}  $ is an anti-derivative of $s \mapsto \frac{1}{\sqrt{s}^3}$, and that the function $s \mapsto 2\sqrt{s}  $ is an anti-derivative of $s \mapsto 1/\sqrt{s}$, we can write that
\begin{equation} \label{majo} r(t) \leq C_{T_1,\lambda,\mu,\alpha_0} +2\sqrt{\frac{\lambda \alpha_0}{\mu}t}, \end{equation}
so by using \eqref{mino} and \eqref{majo}, we get that
\[ \frac{r(t)}{2\sqrt{\frac{\lambda \alpha_0}{\mu}t}} \underset{t\to+\infty}{\longrightarrow} 1, \]
hence the result.
\end{proof}

 From expressions \eqref{expression_psi}, \eqref{expression_a} and \eqref{expression_b}, we can calculate the $L^{\infty}$ norm of our Gaussian solutions:
\[ \| \psi(t) \|_{L^{\infty}} = |b(0)| \exp \left( \frac{1}{2} \int_0^t \Im(a(s)) ds \right) = \frac{|b(0)|}{\sqrt{r(t)}} \sim \sqrt{2}|b(0)| \left( \frac{\mu}{\lambda \alpha_0 t} \right)^{\frac{1}{4}}. \]

In order to express the $H^1$ norm of $\psi$, by the following computations,
\begin{equation*}
\| \nabla \psi(t) \|^2_{L^2}  =  \frac{1}{2} \sqrt{\pi} \frac{|b(t)|^2}{\sqrt{\Re(a(t))}} \frac{|a(t)|^2}{\Re(a(t))} 
%                    =   \frac{1}{2} \sqrt{\pi} \frac{|b(0)|^2}{\cancel{r(t)}} \frac{\cancel{r(t)^2}}{\alpha_0} \frac{\cancel{r(t)}}{\sqrt{\alpha_0}} \left( \frac{\alpha_0^2}{\cancel{r(t)^4}}+\frac{\dot{r}(t)^2}{\cancel{r(t)^2}} \right)  
                    =  \frac{1}{2} \sqrt{\pi} \frac{|b(0)|^2}{\alpha_0 \sqrt{\alpha_0}} \left( \frac{\alpha_0^2}{r(t)^2} +\dot{r}(t)^2 \right), 
\end{equation*}
we see that we also need to have an equivalent of $\dot{r}$, which is the aim of the following proposition:

\begin{proposition} \label{prop_dot_r_defocusing}
As  $t \rightarrow + \infty$, we have
\[ \dot{r}(t) \sim \sqrt{\frac{\lambda \alpha_0}{\mu t}}.  \]
\end{proposition}
\begin{proof}
We already know from the previous proposition that there exists a generic constant $C>0$ such that
\[ r(t) \leq C +2\sqrt{\frac{\lambda \alpha_0}{\mu}t}, \]
and also that $\ddot{r}(t) \leq 0$ for all $t \geq T_1$, so using equation \eqref{eqr} we get that
\[  \mu \dot{r}(t) \geq \frac{\alpha_0^2}{\left( C+2\sqrt{\frac{\lambda \alpha_0 t}{\mu}} \right)^3} + \frac{2 \lambda \alpha_0}{C+2\sqrt{\frac{\lambda \alpha_0 t}{\mu}} } \geq \frac{\lambda \alpha_0}{ \sqrt{\frac{\lambda \alpha_0 t}{\mu}}}, \]
hence
\[ \dot{r}(t) \geq \sqrt{\frac{\lambda \alpha_0}{\mu t}}. \]
In order to get an upper bound, we recall that
\[ r(t) \geq 2 \sqrt{\frac{\lambda \alpha_0 t}{\mu}}. \]
Then, still using equation \eqref{eqr}, we get that
\[ \ddot{r}+\mu \dot{r} = \frac{ \alpha_0^2}{r^3} + \frac{ 2 \lambda \alpha_0}{r} \leq \frac{\sqrt{\alpha_0 \mu }}{8\sqrt{ \lambda  t}^3} +\sqrt{\frac{\lambda \alpha_0 \mu }{t}} =: g(t). \]
We can rewrite the previous inequality as
\[ \frac{d}{dt} \left( \dot{r}(t) e^{\mu t} \right) \leq g(t) e^{\mu t}, \]
so we get
\[ \dot{r}(t) \leq \dot{r}(0)e^{-\mu t} + e^{-\mu t}\int_0^t g(s) e^{\mu s} ds  .\]
Finally, thanks to the asymptotic expansion
\[  \int_0^t \frac{1}{\sqrt{s}} e^{\mu s} ds = \frac{1}{\mu} e^{\mu t} \left( \frac{1}{\sqrt{t}} + o \left( \frac{1}{\sqrt{t}} \right) \right) ,\]
we get the estimate
\[ \dot{r}(t) \leq \sqrt{\frac{\lambda \alpha_0}{\mu t}}+ o \left( \frac{1}{\sqrt{t}} \right)  \]
which ends the proof.
\end{proof}

Thanks to the previous proposition, we finally get
\[ \| \nabla \psi(t) \|_{L^2} \sim  C_{\alpha_0,\mu,\lambda} |b(0)| \frac{1}{\sqrt{t}}.  \]

\begin{remark} \label{remark_d}
All the results above can easily be generalized to any dimension $d$. In fact, if we now consider Gaussian functions of the form
\begin{equation}
 \psi(t,x) = b(t) \exp\left(-\frac{1}{2} \sum_{j=1}^d a_j(t) x_j^2 \right), \label{expression_psi_d}
 \end{equation}
 then plugging it into \eqref{NMeq} gives the system
 \begin{align}
    i \dot{b} - \frac{1}{2} b \sum_{j=1}^d a_j  = \lambda b \log(|b|^2)- \frac{1}{2} i\mu b  \log\left(\frac{b}{b^*} \right), \label{eqb1_d} \\
    i \dot{a_j}- a_j^2  = 2 \lambda \Re(a_j) + \mu \Im(a_j), \ \ \ j=1,...,d \label{eqa1_d}
\end{align}   
which is the $d$-dimensional analogue of system $\eqref{eqb1}-\eqref{eqa1}$. Using equation \eqref{eqb1_d} we get the explicit formulation of $b$ with respect to $(a_j)_{1 \leq j \leq d}$:
\begin{equation}
|b(t)|=|b(0)| \exp \left( \frac{1}{2} \sum_{j=1}^d\Im A_j(t) \right) \label{expression_b_d}
\end{equation}
where
\[A_j(t) := \int_0^t a_j(s) ds.\]

As the $d$ equations \eqref{eqa1_d} are decoupled as $j$ varies, we are now back to the study of equation \eqref{eqr} on each independent dimension component.

\end{remark}

\section{Long-time behavior}

\subsection{Focusing case}

Before going for the proof of Theorem \ref{theorem1}, a first look at the equation, as suggested in \cite{chavanis2017}, would be the study of time-independent solutions of equation \eqref{NMeq} of the form
\begin{equation} \label{stationary}
    \psi(t,x)=f(x) e^{iS(t)}.
\end{equation}It is already well known since 1976 \cite{birula1976} (see also \cite{cazenave1983}) that standing waves of the form $f(x) e^{i \omega t}$ are the stationary solutions for the logarithmic Schrödinger equation, where $\omega \in \R$ and $f \in W(\R^d)$ stands as a solution of the semilinear elliptic equation
\begin{equation} \label{elliptic}
    - \frac{1}{2} \Delta f + \omega f + \lambda f \log |f|^2 =0.
\end{equation}
In our case, dealing with the dissipation term $\log( \psi/\psi^*)$, standing waves are no longer relevant, and stationary solutions are rather of the form \eqref{stationary} with 
\begin{equation} \label{stationary_phase}
    S_{\omega}(t)= \frac{\omega}{\mu} + e^{-\mu t}
\end{equation}  
and $f$ still a solution of \eqref{elliptic} belonging to $W(\R^d)$. Moreover, it is already known (see \cite{birula1976}) in the case $\lambda <0$ that the Gausson
\begin{equation} \label{gausson}
    \phi_{\omega}(x):= e^{\frac{\omega}{2\lambda}+d} e^{\lambda |x|^2}
\end{equation}
solves equation \eqref{elliptic} for any dimension $d$, and up to translations, it is the unique strictly positive $\mathcal{C}^2$-solution for \eqref{elliptic} such that $f(x) \rightarrow 0$ as $|x| \rightarrow \infty$ (see \cite{ardila2016}).\\
Ardila recently showed in \cite{ardila2016} the orbital stability of the Gausson by minimizing some energy functionals. As suggested by the study of the Gaussian case, we will show here that every smooth solution of our system \eqref{continuity}-\eqref{fluid} tends to the Gausson at large times, using both energy minimization and dynamical system theory, as well as the additional regularity and tightness assumptions made in Section 2. \\

% First off, let's give the heuristic behind this result. The presence of the $\partial_t J + \mu J$ term in equation \eqref{fluid} and the look of $S$ in \eqref{stationary_phase} suggests us to define a new phase $\tilde{S}$ defined by
% \begin{equation} \label{new_phase}
%     S(t,x)=\tilde{S}(t,x) e^{-\mu t}.
% \end{equation} 
% Plugging \eqref{new_phase} into \eqref{continuity}-\eqref{fluid}, and denoting $J=\tilde{J}e^{-\mu t}$, we get :
% \begin{gather}
% \partial_t \rho + e^{-\mu t}\Div (\tilde{J}) =0 \label{new_continuity}  \\
% e^{-\mu t} \partial_t \tilde{J} + e^{-2\mu t}\Div \left( \frac{\tilde{J} \otimes \tilde{J}}{\rho} \right) + \lambda \nabla \rho= \frac{1}{2} \rho \nabla \left( \frac{\Delta \sqrt{\rho}}{\sqrt{\rho}} \right). \label{new_fluid}
% \end{gather} 
% Performing the formal limit $t \rightarrow \infty$ in the above system, equation \eqref{new_continuity} gives $\partial_t \rho =0$ and equation \eqref{new_fluid} leads to the equation 
% \[ \lambda \nabla \rho= \frac{1}{2} \rho \nabla \left( \frac{\Delta \sqrt{\rho}}{\sqrt{\rho}} \right)  .\]
% Dividing by $\rho$ and integrating over $\R^d$, we get
% \[ \lambda \log \rho= \frac{1}{2} \frac{\Delta \sqrt{\rho}}{\sqrt{\rho}} +\omega,  \]
% with $\omega \in \R$, which corresponds to equation \eqref{elliptic} denoting $\varphi= \sqrt{\rho}$. Of course the lack of regularity on $\rho$ and the insufficient estimates on quantities like $\Div (\tilde{J})$, $\partial_t \tilde{J}$ or $\Div \left( \tilde{J} \otimes \tilde{J}/\rho \right)$ prevent us from making this formal calculus more rigorous.

 \textit{Proof of Theorem \ref{theorem1}.} We recall that $(\rho,u)$ is an energy weak solution of system \eqref{continuity}-\eqref{fluid}, and we also assume that $\sqrt{\rho} \in L^{\infty}(\R_+, H^2(\R^d))$. We remark that the dissipation term in \eqref{energy} being a multiple of the kinetic energy, we have the following differential equation:
\begin{equation}
    \dot{E}_c(t)+\dot{E}_p(t)=-2 \mu E_c(t) \leq 0.
\end{equation}
By LaSalle’s invariance principle \cite{lasalle1968}, the $\omega$-limit set of any solution of \eqref{eq_energy} is thus reduced to the set
\[  \Omega= \enstq{(\rho, u)}{\dot{E}_c(\rho,u)+\dot{E}_p(\rho,u)=0} = \enstq{(\rho, u)}{E_c(\rho,u)=0} \]
by the previous equation.\\
Denote $(\rho, u)$ a solution to \eqref{continuity}-\eqref{fluid}. Taking a sequence $t_n \rightarrow +\infty$,  we denote 
\[ \rho_n(t,x)=\rho(t+t_n,x) \ \ \ \text{and} \ \ \ u_n(t,x)=u(t+t_n,x) \]
for all $t \in (0,1)$ and $x \in \R^d$. From \eqref{mass_conservation} we get that $(\rho_n)_n$ is bounded in $L^1(\R^d)$, and from \eqref{eq_energy} we get that for all $t \geq 0$,
\[ \sup_{n} \int_{\R^d} \rho_n |\log( \rho_n)|  < \infty.   \]
In fact, writing
\[ \int_{\R^d} \rho_n |\log( \rho_n)|  = \int_{\rho_n > 1} \rho_n |\log( \rho_n)|  - \int_{\rho_n < 1} \rho_n |\log( \rho_n)| ,  \]
we get that
\[ \int_{\rho_n > 1} \rho_n |\log( \rho_n)| \leq \int_{\R^d} \sqrt{\rho_n}^{2+\eps}  \lesssim \| \sqrt{\rho_n} \|_{L^2(\R^d)}^{(2+\eps)(1-\alpha)} \| \nabla \sqrt{\rho_n} \|_{L^2(\R^d)}^{(2+\eps)\alpha}
 \lesssim \| \nabla \sqrt{\rho_n} \|_{L^2(\R^d)}^{\frac{\eps}{2}} 
 \]
by Gagliardo-Nirenberg inequality with $\alpha = 1/2 - 1/(2+\eps) >0$, for small $\epsilon > 0$, so $\int \rho_n |\log( \rho_n)| $ is uniformly bounded using the energy estimate \eqref{eq_energy}. Then, by de la Vallée-Poussin theorem (see \cite{meyer1966}),  $(\rho_n)_n$ is equi-integrable, and so it converges weakly (up to a subsequence) by Dunford-Pettis theorem (see e.g.\cite{dunford1938}):
\[ \rho_n \underset{n\to+\infty}{\rightharpoonup} \rho_{\infty} \ \ \ \text{weakly in} \ L^1(\R^d).  \]
Furthermore, using Assumption \ref{assumption_tightness} we get that the sequence $(\rho_n)_n$ is tight in $L^1(\R^d)$, so
\begin{equation} \label{L1_rho_infty}
    \| \rho_{\infty} \|_{L^1(\R^d)}=\| \rho_0 \|_{L^1(\R^d)}.
\end{equation}
From LaSalle’s invariance principle we get that $E_c( \rho_n, u_n) \rightarrow 0$, $ie$
\[  \| \sqrt{\rho_n}u_n \|_{L^2(\R^d)} \underset{n\to+\infty}{\longrightarrow} 0. \]

Multiplying equation \eqref{continuity} by a test function $\varphi \in \mathcal{D}( (0,1) \times \R^d)$ and integrating over space and time, we get by integrating by parts
\[ \int_0^1 \int_{\R^d} \left( \rho_n \partial_t \varphi + \rho_n u_n \cdot \nabla \varphi \right) dx dt =0.  \]
By dominated convergence using equation \eqref{mass_conservation} and from weak convergence of $(\rho_n)_n$ we get that
\[  \int_0^1 \int_{\R^d} \rho_n \partial_t \varphi dx dt \underset{n\to+\infty}{\longrightarrow} \int_0^1 \int_{\R^d} \rho_{\infty} \partial_t \varphi . \]
Furthermore, writing $\rho_n u_n = \sqrt{\rho_n} \sqrt{\rho_n} u_n$, we get, still by dominated convergence using equation \eqref{energy}, that
\[  \left| \int_0^1 \int_{\R^d}\sqrt{\rho_n} \sqrt{\rho_n} u_n \cdot \nabla \varphi dx dt \right| \leq C_{\varphi} \| \sqrt{\rho_0} \|_{L^2(\R^d)} \int_0^1 \| \sqrt{\rho_n}u_n \|_{L^2(\R^d)} dt \underset{n\to+\infty}{\longrightarrow} 0. \]
So we have that in the distribution sense,
\begin{equation}
  \partial_t \rho_{\infty} =0.  \label{continuity_infty} 
\end{equation}
Now multiplying equation \eqref{fluid} by a test function $\varphi \in \mathcal{D}( (0,1) \times \R^d)$ and integrating over space and time, we get by integrating by parts:
\begin{multline*}
 \int_0^1 \int_{\R^d} \left( \rho_n u_n \partial_t \varphi + \rho_n u_n \otimes u_n \nabla \varphi + \lambda \rho_n \Div \varphi - \mu \rho_n u_n \varphi \right) dx dt \\
 = \int_0^1 \int_{\R^d} \left( \frac{1}{4} \rho_n \Div \Delta  \varphi - \nabla \sqrt{\rho_n} \otimes \nabla \sqrt{\rho_n} \nabla \varphi \right) dx dt.
 \end{multline*}
Denoting $\rho_n u_n \otimes u_n = \sqrt{\rho_n} u_n \otimes \sqrt{\rho_n} u_n$, by the same arguments as above, we get that
\[  \int_0^1 \int_{\R^d} \left( \rho_n u_n \partial_t \varphi + \sqrt{\rho_n} u_n \otimes \sqrt{\rho_n} u_n \nabla \varphi + \lambda \rho_n \Div \varphi - \mu \rho_n u_n \varphi \right) dx dt \underset{n\to+\infty}{\longrightarrow} \int_0^1 \int_{\R^d} \lambda \rho_{\infty} \Div \varphi. \] 
From weak convergence of $(\rho_n)_n$ we get that 
\[ \int_0^1 \int_{\R^d} \rho_n \Div \Delta  \varphi dx dt \underset{n\to+\infty}{\longrightarrow}   \int_0^1 \int_{\R^d} \rho_{\infty} \Div \Delta  \varphi dx dt,\]
so the only remaining term is
\[  \int_0^1 \int_{\R^d} \nabla \sqrt{\rho_n} \otimes \nabla \sqrt{\rho_n} \nabla \varphi  dx dt \underset{n\to+\infty}{\longrightarrow} \int_0^1 \int_{\R^d} \nabla \sqrt{\rho_{\infty}} \otimes \nabla \sqrt{\rho_{\infty}} \nabla \varphi  dx dt .\]
In fact, we recall that from \eqref{energy} we get that $\nabla \sqrt{\rho_n}$ is uniformly bounded in $L^2$, and so
\[ \nabla \sqrt{\rho_n} \underset{n\to+\infty}{\rightharpoonup}  \nabla \sqrt{\rho_{\infty}}  \ \ \ \text{weakly in} \ L^2(\R^d).   \]
As $f \mapsto \int_0^1 \int_{\R^d} f \otimes \nabla \sqrt{\rho_{\infty}} \nabla \phi$ is a linear form on $L^2(\R^d)$, we have
\[ \left|\int_0^1 \int_{\R^d} \left(\nabla \sqrt{\rho_n} -\nabla \sqrt{\rho_{\infty}}  \right)\otimes \nabla \sqrt{\rho_{\infty}} \nabla \varphi  dx dt \right| \underset{n\to+\infty}{\longrightarrow} 0. \]
Then, using \eqref{energy} and the additional assumption that $\sqrt{\rho_n} \in L^{\infty}(\R_+, H^2(\R^d))$, we get that
\begin{multline*}
\left|\int_0^1 \int_{\R^d} \left(\nabla \sqrt{\rho_n} -\nabla \sqrt{\rho_{\infty}}  \right)\otimes \nabla \sqrt{\rho_{\infty}} \nabla \varphi  dx dt \right| \\ 
\leq \|\nabla \sqrt{\rho_{\infty}}\|_{L^2(\R^d)}  \int_0^1 \int_{\R^d} \left| \nabla \sqrt{\rho_n} -\nabla \sqrt{\rho_{\infty}}  \right|^2  |\nabla \varphi|^2 dx dt   \underset{n\to+\infty}{\longrightarrow} 0, 
\end{multline*}
since $H^2(\R^d)$ is compactly embedded into $H^1_{\loc}(\R^d)$, and recalling that $|\nabla \varphi|^2$ has a compact support, so finally
\[ \left|\int_0^1 \int_{\R^d} \nabla \sqrt{\rho_n} \otimes \nabla \sqrt{\rho_n} \nabla \varphi  dx dt -  \int_0^1 \int_{\R^d} \nabla \sqrt{\rho_{\infty}} \otimes \nabla \sqrt{\rho_{\infty}} \nabla \varphi  dx dt \right| \underset{n\to+\infty}{\longrightarrow} 0, \]
so we have the result. Under Assumption \ref{assumption_tightness}, we now get that $\rho_n$ converges weakly in $L^1$ to $\rho_{\infty}$ which is solution of equation \eqref{continuity_infty} and
\begin{equation}
 \lambda \nabla \rho_{\infty} = \frac{1}{2} \rho_{\infty} \nabla \left( \frac{\Delta \sqrt{\rho_{\infty}}}{\sqrt{\rho_{\infty}}} \right).  \label{fluid_infty}
\end{equation} 
As mentioned in \cite{jungel2012}, assuming $\rho_{\infty} >0$, we can divide equation \eqref{fluid_infty} by $\rho_{\infty}$. Then, integrating over $\R^d$ and multiplying by $f:=\sqrt{\rho_{\infty}}$, equation \eqref{fluid_infty} is linked to the already mentioned second-order nonlinear elliptic equation:
\[ -\frac{1}{2}\Delta f + \omega f +  \lambda f \log |f|^2 =0, \]
where $\omega$ denotes the integration constant. From \cite{davenia2014} we know that every solution $f$ of \eqref{elliptic} such that $f \in L^1_{\loc}(\R^d)$ and $\Delta f \in L^1_{\loc}(\R^d)$ in the sense of distribution is either trivial or strictly positive. \\
In our case, $\rho_{\infty} >0$ so by standard elliptic regularity arguments, $f$ is $\mathcal{C}^2(\R^d)$, and from \cite{davenia2014} we infer that, up to translations, $f$ is equal to the Gausson \eqref{gausson}. Hence for all $(t_n)_n$, the sequence $(\rho_n)_n$ has the same limit, so we get that 
\begin{equation} \label{weak_convergence_rho}
 \rho(t,.) \underset{t\to+\infty}{\rightharpoonup} \rho_{\infty} \ \ \ \text{weakly in} \ L^1(\R^d),  
\end{equation}
with $\rho_{\infty} = c_{\lambda} e^{\lambda|x|^2}$ up to translations. Furthermore, from the conservation of mass \eqref{L1_rho_infty} we get that $c_{\lambda}=\| \rho_0 \|_{L^1(\R^d)} (-\lambda/\pi)^{d/2}$. We now denote by $x_{\infty}$ the (unique) center of the Gaussian function $\rho_{\infty}$, and for all $R >0$, we call $\xi_R \in \mathcal{C}(\R^d)$ the following smooth function such that $\xi_R =1$ on $\left\{ |x| \leq R \right\}$ and $\xi_R =0$ on $\left\{ |x| \geq 2R \right\}$. In order to determine $x_{\infty}$, we use the weak convergence \eqref{weak_convergence_rho} to infer that, $\forall R >0$,
\[ \int_{\R^d} x \xi_R(x) \rho(t,x)dx \underset{t\to+\infty}{\longrightarrow} \int_{\R^d} x \xi_R(x) \rho_{\infty}(x)dx.   \]
As $\rho_{\infty}$ is a Gaussian function, we get that
\[ \int_{\R^d} x \xi_R(x) dx =  c_{\lambda} \int_{\R^d} x e^{\lambda|x-x_{\infty}|^2}dx  \underset{R\to+\infty}{\longrightarrow} x_{\infty} \| \rho_0\|_{L^1(\R^d)}. \]
What's more, as we know that $\int_{\R^d} |x|^2 \rho(t,x)dx < \infty$ for all $t\geq 0$ by Assumption \ref{assumption_tightness}, we get by dominated convergence theorem that
\[ \int_{\R^d} x  \rho(t,x)dx \underset{t\to+\infty}{\longrightarrow} x_{\infty} \| \rho_0\|_{L^1(\R^d)}, \]
which concludes the proof.

\subsection{Defocusing case}
We are now going to prove Theorem \ref{theorem2}. Following \cite{carles2018}, we define $\phi$ according to the formal solution $\psi$ of \eqref{NMeq}:
\begin{equation} \label{exp_phi}
    \psi(t,x)=\frac{1}{\tau(t)^{d/2}} \phi \left( t,\frac{x}{\tau(t)} \right) \frac{\|\psi_0 \|_{L^2(\R^d)}}{ \|\gamma \|_{L^2(\R^d)} } \exp \left( i \frac{\dot{\tau}(t)}{\tau(t)} \frac{|x|^2}{2} \right),
\end{equation}
where $\gamma := e^{-|x|^2/2}$ denotes the usual centered Gaussian function and $\tau$ is now the unique solution of the differential equation 
\begin{equation*}
\ddot{\tau} = \frac{ 2 \lambda}{\tau} - \mu \dot{\tau}, \ \ \ \tau(0)=1, \ \ \ \dot{\tau}(0)=0.
\end{equation*}
Mimicking the proof of Proposition \ref{prop_r_defocusing} and Proposition \ref{prop_dot_r_defocusing} (with $\alpha_0 =1$), we also get that
\[  \tau(t) \sim 2\sqrt{\frac{\lambda }{\mu}t} \ \ \ \text{and} \ \ \ \dot{\tau}(t) \sim \sqrt{\frac{\lambda }{\mu t}}. \]
In fact, the previous study of equation \eqref{eqr} in the defocusing case shows that the $\alpha_0^2/r^3$ has no influence in the long-time behavior of the solution $r$, so we can get rid of it in the forthcoming calculations. Note that this last remark is no longer true in the focusing case, where the $\alpha_0^2/r^3$ term plays a crucial role in the asymptotic behavior of the solution. \\

Plugging \eqref{exp_phi} into \eqref{NMeq}, we get that $\phi$ satisfy the equation 
\begin{equation*}  i  \partial_t \phi + \frac{1}{2\tau^2}  \Delta \phi = \lambda \phi \log \left( \left|\frac{\phi}{\gamma} \right|^2 \right) + \frac{1}{2i}  \mu  \phi \log\left( \frac{\phi}{\phi^*} \right) -d \lambda \phi \log \tau + 2 \lambda \phi \log \left( \frac{\|\psi_0 \|_{L^2(\R^d)}}{ \|\gamma \|_{L^2(\R^d)} } \right),  \end{equation*} 
where the initial datum for $\phi$ is
\begin{equation} \label{scaled_initial_datum}
    \phi(0)=\phi_0:= \frac{ \|\gamma \|_{L^2(\R^d)} }{\|\psi_0 \|_{L^2(\R^d)}} \psi_0.
\end{equation} 
Using the gauge transform consisting in replacing $\phi$ with $\phi e^{-i \theta(t)}$, where $\theta$ is the unique solution of the linear differential equation 
\[    \dot{\theta} + \mu \theta =-d \lambda \log \tau -2 \lambda \log \phi_0 , \]
we may assume that the last two terms are absent, and we look into the following scaled equation:
\begin{equation}  i  \partial_t \phi + \frac{1}{2\tau^2}  \Delta \phi = \lambda \phi \log \left( \left|\frac{\phi}{\gamma} \right|^2 \right) + \frac{1}{2i}  \mu  \phi \log\left( \frac{\phi}{\phi^*} \right) .  \label{NM_phi_eq} \end{equation} 

Of course all the above formal calculations rigorously stand in the fluid case, so the hydrodynamical formulation of \eqref{NM_phi_eq} is the following:
\begin{subequations} \label{hydro}
\begin{empheq}[left=\empheqlbrace]{align}
& \partial_t \varrho + \frac{1}{\tau^2}\Div \mathcal{J} =0, \label{continuity_phi}  \\
& \partial_t \mathcal{J} + \frac{1}{\tau^2} \Div \left( \frac{\mathcal{J} \otimes \mathcal{J}}{\varrho} \right) + \lambda \nabla \varrho +2 \lambda y \varrho + \mu \mathcal{J}= \frac{1}{2 \tau^2} \varrho \nabla \left( \frac{\Delta \sqrt{\varrho}}{\sqrt{\varrho}} \right), \label{fluid_phi}
\end{empheq} 
\end{subequations}
where $y=x/\tau$. The couple $(\varrho,\mathcal{J})$ induced by Assumption \ref{assumption_reg_defocusing} and the above scaling will be called energy weak solution of \eqref{continuity_phi}-\eqref{fluid_phi}, according to Definition \ref{weak}.\\ \\
Let us now explain the heuristic behind the calculations of the forthcoming proof about the long time behavior of our solutions. Letting $\tau(t) \rightarrow \infty$, equation \eqref{fluid_phi} simplifies into 
\begin{equation} \label{fluid_phi_simp}
 \partial_t \mathcal{J} + \lambda \nabla \varrho +2 \lambda y \varrho + \mu \mathcal{J}= 0, 
 \end{equation}
then differentiating equation \eqref{continuity_phi} in time,
\begin{equation} \label{Div_continuity_phi}
\Div \partial_t \mathcal{J} = - \partial_t ( \tau^2 \partial_t\varrho),   
\end{equation}
and plugging it into \eqref{fluid_phi_simp}, we get
\begin{equation*} 
 - \partial_t ( \tau^2 \partial_t\varrho) - \mu \tau^2 \partial_t \varrho + \lambda \Delta \varrho +2 \lambda \Div(y \varrho )= 0. 
 \end{equation*}
 As 
 \[ \partial_t ( \tau^2 \partial_t\varrho) = \tau^2 \partial_t^2 \varrho + 2\tau \dot{\tau} \partial_t \varrho, \]
 and recalling that $ \tau \sim 2 \sqrt{\lambda t/\mu} $ and $ \dot{\tau}  \sim 2 \sqrt{\lambda /(\mu t)}$, we get that $\tau \dot{\tau} \rightarrow 2$ and $\tau^2 \ll  (\tau^2)^2$ as $t \rightarrow + \infty$, so we can neglect the $\partial_t^2 \varrho$ term. Introducing another scaling in time $s$ defined by
 \begin{equation} \label{scaling}
 \partial_s = \frac{\mu}{\lambda} \tau^2 \partial_t,
 \end{equation}
 and discarding the lower-order terms, we finally find that
 \[ \partial_s \varrho = L \varrho,   \]
 where $L:=\lambda \Delta  +2 \lambda \Div(y . )$ denotes a Fokker-Planck operator, for which it is well known (under some hypothesis that we will recall in the following) that in large times the solution converges strongly to a Gaussian. Note that unlike the elliptic equation \eqref{elliptic} of the focusing case, the Fokker-Planck operator $L$ is not invariant by translations, hence the limit will be unique (and not only unique up to translations).

\begin{remark}
Note that the definition  \eqref{scaling} of $s$ can be rewritten
\[ s = \frac{\mu}{\lambda} \int \frac{1}{\tau^2}, \]
that gives $s \sim \log(t)/4$. This scaling is different from the one for the logarithmic Schrödinger equation described in \cite{carles2018}(which was of order $s \sim \log \log(t)/4$). However, we will see that the same proof stands in order to get the long-time dynamics of our solutions.
\end{remark}

In order to perform the limit in a rigorous way, we will first look at the energy quantities of our system \eqref{continuity_phi}-\eqref{fluid_phi}. The scaled form of the energy inequality \eqref{eq_energy} referring to system \eqref{continuity_phi}-\eqref{fluid_phi} is the following:
\begin{equation} \mathcal{E}(t)  +  \mu \int_0^t \mathcal{D}(s) ds \leq \mathcal{E}_0, \label{scaled_eq_energy} \end{equation}
where
\begin{gather}
   \mathcal{E}(t) =\frac{1}{2 \tau(t)^2} \mathcal{E}_{\kin}(t) + \lambda \mathcal{E}_{\ent}(t), \label{scaled_energy} \\
    \mathcal{E}_{\kin}(t)=  \frac{1}{2} \int_{\R^d} \left(  \varrho |u|^2 +  |\nabla \sqrt{\varrho} |^2  \right) dy, \label{scaled_energy_kin} \\
    \mathcal{E}_{\ent}(t)=\int_{\R^d}  \varrho\log \left( \frac{\varrho}{\Gamma}  \right) dx \label{scaled_energy_ent} \\
    \mathcal{D}(s)= \int_{\R^d} \varrho |u|^2 dx, \label{scaled_dissipation}
\end{gather}
$\mathcal{E}_0:=E_0$ and $\Gamma := e^{-|x|^2}=\gamma^2$. Now we recall that from \eqref{scaled_initial_datum} and \eqref{continuity_phi}, for all $t \geq 0$,
\begin{equation*} \label{scaled_mass_conservation}
    \| \varrho(t) \|_{L^1(\R^d)} = \| \varrho_0 \|_{L^1(\R^d)} =\| \Gamma \|_{L^1(\R^d)}. 
\end{equation*}
As $\varrho,\Gamma \geq 0$, we can apply the Csiszár-Kullback inequality (see e.g. \cite{ineg_sobo_log}) that gives
\[ \mathcal{E}_{\ent}(t) \geq \frac{1}{2 \| \Gamma \|_{L^1}} \| \varrho(t)-\Gamma \|^2_{L^1} \geq 0.   \]
We now state the first lemma of this section, which will be useful in the following to get some convergence estimates:
\begin{lemma} \label{lemma_entropy}
There holds
\begin{equation} \label{equation_lemma_1}
    \sup_{t\geq 0} \left( \int_{\R^d}  \varrho(t,y)(1 +|y|^2 +|\log \varrho(t,y) |)  dy + \frac{1}{ \tau(t)^2} \mathcal{E}_{\kin}(t) \right) < \infty
\end{equation}
and
\begin{equation} \label{equation_lemma_2}
    \int_0^{\infty} \frac{\dot{\tau}(t)}{\tau^3(t)} \mathcal{E}_{\kin}(t) dt  < \infty.
\end{equation}
\end{lemma}
\begin{proof}
Denoting $\mathcal{E}_+$ the positive part of the scaled energy $\mathcal{E}$, we have
\[ \mathcal{E}_+ = \frac{1}{2 \tau^2} \mathcal{E}_{\kin} + \lambda \int_{\varrho > 1} \varrho \log \varrho + \lambda \int_{\R^d} |y|^2 \varrho,  \]
so equation \eqref{scaled_eq_energy} gives
\[ \mathcal{E}_+(t) + \mu \int_0^t \mathcal{D}(s) ds \leq  \mathcal{E}_0 + \lambda \int_{\varrho < 1} \varrho \log\left(\frac{1}{\varrho} \right) . \]
The last term is controlled by
\[  \int_{\varrho < 1} \varrho \log\left(\frac{1}{\varrho} \right) \lesssim \int_{\R^d} \varrho^{1-\eps}  \]
for all $\eps >0$ arbitrary small. By an interpolation formula, we get that
\[ \int_{\R^d} \varrho^{1-\eps}  \lesssim \| \varrho \|_{L^1}^{1-\eps\frac{d +2}{2} }  \| |y|^2 \varrho \|_{L^1}^{ \frac{d \eps}{2} } \]
for all $0< \eps < 2/(d+2)$. This implies that for all $t \geq 0$, 
\[ \mathcal{E}_+(t) + \mu \int_0^t \mathcal{D}(s) ds \lesssim  1 + \mathcal{E}_+^{\frac{d \eps}{2}}(t),  \]
thus $\mathcal{E}_+(t) \in L^{\infty}(\R_+)$, and  equation \eqref{equation_lemma_1} follows. Differentiating $\mathcal{E}$, we get
\[ \dot{\mathcal{E}}= -2 \frac{\dot{\tau}}{\tau} \mathcal{E}_{\kin},  \]
and since $\mathcal{E}(t) \geq 0$, \eqref{equation_lemma_2} follows from \eqref{equation_lemma_1}.
\end{proof}

\begin{proposition}
Let $(\varrho,\mathcal{J})$ be a smooth solution of the scaled hydrodynamical system \eqref{continuity_phi}-\eqref{fluid_phi}. Then 
\[ \varrho(t) \underset{t \to \infty}{\rightharpoonup} \Gamma  \text{ weakly in } L^1(\R^d). \]
\end{proposition}
\textit{Proof of Theorem \ref{theorem1}.}
By the elimination of $\mathcal{J}$ described above, using equation \eqref{Div_continuity_phi}, we get 
\begin{equation*} 
 - \partial_t ( \tau^2 \partial_t\varrho) - \mu \tau^2 \partial_t \varrho + \lambda L \varrho = \frac{1}{2 \tau^2} \varrho \nabla \left( \frac{\Delta \sqrt{\varrho}}{\sqrt{\varrho}} \right)-\frac{1}{\tau^2} \Div \left( \frac{\mathcal{J} \otimes \mathcal{J}}{\varrho} \right) . 
 \end{equation*}
 Using the change of variables \eqref{scaling} and introducing the notation
 \[ \tilde{\varrho}(s(t),y):= \varrho(t,y), \]
 we calculate the quantities
 \[ \partial_t ( \tau^2 \partial_t\varrho) = \frac{\lambda^2}{\mu^2}  \frac{1}{\tau^2} \partial_s^2 \tilde{\varrho} \ \text{and} \  \mu \tau^2 \partial_t \varrho = \lambda \partial_s \tilde{\varrho},  \]
hence we obtain the following equation:
 \begin{equation} \label{equation_s}
    -  \frac{\lambda^2}{\mu^2}  \frac{1}{\tau^2} \partial_s^2 \tilde{\varrho} -\lambda \partial_s \tilde{\varrho} + \lambda L \tilde{\varrho} = \frac{1}{2 \tau^2} \tilde{\varrho} \nabla \left( \frac{\Delta \sqrt{\tilde{\varrho}}}{\sqrt{\tilde{\varrho}}} \right)-\frac{1}{\tau^2} \Div \left( \frac{\mathcal{J} \otimes \mathcal{J}}{\tilde{\varrho}} \right) . 
 \end{equation}
Now we remark that equation  \eqref{equation_lemma_1} induces
\begin{equation} \label{scaled_equation_lemma_1}
    \sup_{s\geq 0}  \int_{\R^d} \tilde{\varrho}(s,y) (1 +|y|^2 +|\log \tilde{\varrho}(s,y) |)  dy   < \infty, 
\end{equation} 
\eqref{equation_lemma_2} gives
\begin{equation} \label{scaled_equation_lemma_2}
\int_0^{\infty}\dot{\tau}(s) \mathcal{E}_{\kin}(s) dt  < \infty,  
\end{equation}
and that 
\[   \tau(s) \sim 2 \sqrt{\frac{\lambda}{\mu}} e^{2s} \ \text{and} \ \dot{\tau}(s)\sim \sqrt{\frac{\lambda}{\mu}} e^{-2s}, \]
so we can conclude like in \cite{carles2018}. Let a sequence $s_n \rightarrow \infty$, take $s\in \left[-1,2 \right]$, and denote
\[ \tilde{\varrho}_n(s,y):=\tilde{\varrho}(s+s_n,y).\]
From \eqref{scaled_equation_lemma_1} along with the de la Vallée-Poussin and Dunford-Pettis theorems, we get the following weak convergence (up to a subsequence, not relabeled for reader's convenience), for all $p \in \left[1,\infty \right)$,
\[ \tilde{\varrho}_n  \underset{t \to \infty}{\rightharpoonup} \tilde{\varrho}_{\infty} \ \ \ \text{in} \ L^p(-1,2;L^1(\R^d)). \]
We also get the weak convergence of the initial datum, up to another subsequence:
\[ \tilde{\varrho}_n(0)  \underset{t \to \infty}{\rightharpoonup} \tilde{\varrho}_{0,\infty} \ \ \ \text{in} \ L^1(\R^d). \]
Thanks to \eqref{scaled_equation_lemma_1}, we also get that the family $( \tilde{\varrho}(s_n,.))_n$ is tight, so
\[ \int_{\R^d} \tilde{\varrho}_{0,\infty}(y) dy = \int_{\R^d} \Gamma(y) dy \]
and
\[ \int_{\R^d} \tilde{\varrho}_{0,\infty}(y)(1 +|y|^2 +|\log  \tilde{\varrho}_{0,\infty}(y) |)   dy   < \infty. \]
Then, denoting $\tau_n(s)=\tau(s+s_n)$, equation \eqref{scaled_equation_lemma_2} implies that
\[ \frac{1}{2 \tau_n^2} \tilde{\varrho}_n \nabla \left( \frac{\Delta \sqrt{\tilde{\varrho}_n}}{\sqrt{\tilde{\varrho}_n}} \right)-\frac{1}{\tau_n^2} \Div \left( \frac{\mathcal{J}_n \otimes \mathcal{J}_n}{\tilde{\varrho}_n} \right)  \underset{t \to \infty}{\rightharpoonup} 0 \ \ \ \text{in} \ L^1(-1,2;W^{-2,1}(\R^d)). \]
In addition, in \eqref{equation_s}, all the other terms but two obviously go weakly to zero, which yields
\begin{equation} \label{eq_infty}
    \partial_s \tilde{\varrho}_{\infty} = L \tilde{\varrho}_{\infty} 
\end{equation}
in $\mathcal{D}'((-1,2) \times \R^d)$, with $\tilde{\varrho}_{\infty}(0,;)= \tilde{\varrho}_{0,\infty} \in L^1(\R^d)$. Thanks to the above bounds on $\tilde{\varrho}_{0,\infty}$, it is known (see \cite{arnold2000}) that the  solution $\tilde{\varrho}_{\infty}$ to \eqref{eq_infty} is actually defined for all $s \geq0$ and satisfies 
\begin{equation} \label{conv_gaussian_infty}
    \| \tilde{\varrho}_{\infty} - \Gamma \|_{L^1(\R^d)} \underset{t\to \infty}{\longrightarrow} 0.
\end{equation}
Going back to system \eqref{continuity_phi}-\eqref{fluid_phi}, we need to show that $\tilde{\varrho}_{\infty}$ is independent of $s$. In the $s$ variable, system \eqref{continuity_phi} becomes
\begin{equation} \label{continuity_phi_s} 
\partial_s \tilde{\varrho} + \frac{\mu}{\lambda}\Div \tilde{\mathcal{J}} =0,  
\end{equation}
and \eqref{scaled_equation_lemma_2} implies that $\tilde{\mathcal{J}} \in L^2(-1,2;L^1(\R^d))$. With $\tilde{\mathcal{J}}_n(s):= \tilde{\mathcal{J}}(s=s_n)$, we have
\[ \Div \tilde{\mathcal{J}}_n \underset{n\to \infty}{\longrightarrow} 0 \ \ \ \text{in} \ L^2(-1,2;W^{-1,1}(\R^d)), \]
so
\[ \partial_s \tilde{\varrho}_{\infty}=0. \]
Combining this last equality with equation \eqref{conv_gaussian_infty}, we infer that $\tilde{\varrho}_{\infty}=\Gamma$. The limit being unique, no extraction of a subsequence is needed, and we conclude that 
\[ \tilde{\varrho}(s) \underset{s \to \infty}{\rightharpoonup} \Gamma  \text{ weakly in } L^1(\R^d). \]

\section{Numerical simulations}
In this section we will plot some numerical trajectories of equation \eqref{NMeq}. As performed in \cite{carles_bao_DF} and \cite{carles_bao_splitting} in the logarithmic case $\mu=0$, in order to avoid numerical blow-up and round-off error due to the first logarithmic component of equation \eqref{NMeq}, we are going to discretize the following regularized equation, with a small parameter $\eps >0$:
\[ \left\{  \begin{array}{l} \label{NMeq_reg}
i  \partial_t \psi^{\eps} + \frac{1}{2}  \Delta \psi^{\eps} = \lambda \psi^{\eps} \log(|\psi^{\eps}|^2+\eps) + \frac{1}{2i}  \mu  \psi^{\eps} \log\left( \frac{\psi^{\eps}}{(\psi^{\eps})^*} \right), \ \ \ x \in \Omega, \ \ \ t>0,  \\
\psi^{\eps}(0,x)=\psi_0(x),  \end{array} \right.  \]
where $\Omega = \R^d$ or $\Omega \subset \R^d$ is a bounded domain with homogeneous Dirichlet boundary condition or periodic boundary condition posted on the boundary. We remark that mass conservation still holds, in the sense that
\[ \forall t \geq 0, \ \ \ \| \psi^{\eps}(t,.) \|_{L^2(\R^d)}=\| \psi_0 \|_{L^2(\R^d)}, \]
and the new regularized energy is defined as follows:
\begin{multline*}
 \forall t \geq 0, \ \ \ E^{\eps}(t)= \int_{\Omega} \left( |\nabla \psi^{\eps}(t,x) |^2 + 2 \lambda \eps |\psi^{\eps}(t,x)| + \lambda |\psi^{\eps}(t,x)|^2 \log( |\psi^{\eps}(t,x)|^2 + \eps) \right. \\
 \left. - \lambda \eps^2 \log(1 + |\psi^{\eps}(t,x)|/\eps)^2 \right) dx.
 \end{multline*}
 We will perform a semi-discretization in time with a Lie-Trotter splitting method. The operator splitting methods for the time integration of \eqref{NMeq_reg} are based on the following splitting 
 \[  \partial_t \psi^{\eps} = A( \psi^{\eps} ) + B ( \psi^{\eps} ) + C(\psi^{\eps} ),   \]
 where
 \[ A( \psi^{\eps}) = \frac{1}{2} i \Delta \psi^{\eps}, \ \ \ B ( \psi^{\eps} ) = -i \psi^{\eps} \log( |\psi^{\eps}|^2 + \eps) , \ \ \ C(\psi^{\eps} )= -\frac{1}{2}  \mu  \psi^{\eps} \log\left( \frac{\psi^{\eps}}{(\psi^{\eps})^*} \right), \]
and the solutions of the subproblems
\[ i  \partial_t u(t,x)=- \frac{1}{2}  \Delta u(t,x), \ \ \ u(0,x)=u_0(x), \ \ \ x \in \Omega, \ \ \ t>0,      \]
\[ i  \partial_t v(t,x)= \lambda v(t,x) \log( v(t,x)+\eps) , \ \ \ v(0,x)=v_0(x), \ \ \ x \in \Omega, \ \ \ t>0,      \]
\[ i  \partial_t w(t,x)= \frac{ \mu}{2i}   w(t,x) \log\left( \frac{w(t,x)}{w^*(t,x)} \right), \ \ \ w(0,x)=w_0(x), \ \ \ x \in \Omega, \ \ \ t>0.      \]
The associated operators are explicitly given, for $t \geq 0$, by
\[  u(t,.)=\Phi^t_A(u_0)=e^{it \Delta} u_0, \]
\[  v(t,.)=\Phi^t_B(v_0)= v_0 e^{-it \log( |v_0|^2+\eps)} , \]
\[  w(t,.)=\Phi^t_C(w_0)=a_0 e^{i \theta_0 e^{- \mu t}}, \ w_0=a_0 e^{i \theta_0}. \]
All the numerical simulations will be made in dimension $d=1$ on the interval $\Omega = \left[a ,b \right]$. The computation of $\Phi_A^t$ is made by a Fast Fourier Transformation. Let $N$ be a positive even integer and denote $\Delta x=(b-a)/N$ and the grid points $x_j=a+k \Delta x $ for $0 \leq k \leq N-1$. Denote by $\psi^{N,j}$ the discretized solution vector over the grid $(x_k)_{0 \leq k \leq N-1}$ at time $t=t_j= j \Delta t$. Let $\mathcal{F}_N$ and $\mathcal{F}_N^{-1}$ denote the discrete Fourier transform and its inverse, respectively. With this notation, $\Phi^t_A(\psi^{N,j})$ can be obtained by 
\[ \Phi^t_A(\psi^{N,j}) = \mathcal{F}_N^{-1} \left( e^{-i \Delta t ( \mu^N)^2} \mathcal{F}_N( \psi^{N,j} )    \right),  \]
where
\[  \mu^N = \frac{2 \pi}{b-a} \left[ 0,1, \ldots, \left( \frac{N}{2} -1 \right), -\frac{N}{2}, \ldots,-1\right],  \]
and the multiplication of two vectors is taken as point-wise. \\

We perform our simulations with a space step $\Delta x=0.2$ on the interval $\Omega =\left[-100,100\right]$. We take the saturation constant $\eps=10^{-3}$, and the dissipation constant $\mu=1$, using a time step $\Delta t =10^{-3}$ on the interval $\left[0,T_{\max} \right]$ with $T_{\max}=1000$. We perform two simulations with $\lambda$ either equal to 0.1 or -0.1, with the initial function
\[  \psi_0(x)=|\sin x| e^{-0.1(x-3)^2}+ |\cos x|e^{-0.2(x+4)^2}. \]
In the focusing case, we clearly observe the convergence to a Gaussian function of the same mass (Figure \ref{fig:lambda_neg}). In the defocusing case, our solution vanishes with a Gaussian profile (Figure \ref{fig:lambda_pos}). \\ 

In order to corroborate the tightness hypothesis Assumption \ref{assumption_tightness} made in Section 2 in the focusing case $\lambda <0$, we are also going to test the worst case scenario of two symmetric Gaussian function diverging at both infinities. Here we take the initial function
\[  \varphi_0(x)=e^{-(x-10)^2+100 i x} + e^{-(x+10)^2-100 i x}, \]
so our two Gaussian functions have huge initial velocities trying to make them go to infinity. This is indeed the case for the linear Schrödinger equation (see Figure \ref{fig:2_gaussian_free}). We observe that the dissipation (taking $\lambda =-0.1$ and $\mu=10$ in order to enhance the nonlinear effect) seems to prevent this behavior from happening by quickly freezing their diverging dynamics (cf Figure \ref{fig:2_gaussian_mu}). Here we do not go for large times, as we take $T_{\max}=10$ with $\Delta t=10^{-4}$, in order to avoid edge effect from our moving Gaussian functions in the free case. Other constants are taken as above.

\bibliographystyle{plain}
\bibliography{biblio}

\begin{figure}[p]
	\centering
		\includegraphics[width=0.85\textwidth,trim = 0cm 1cm 0cm 1cm, clip]{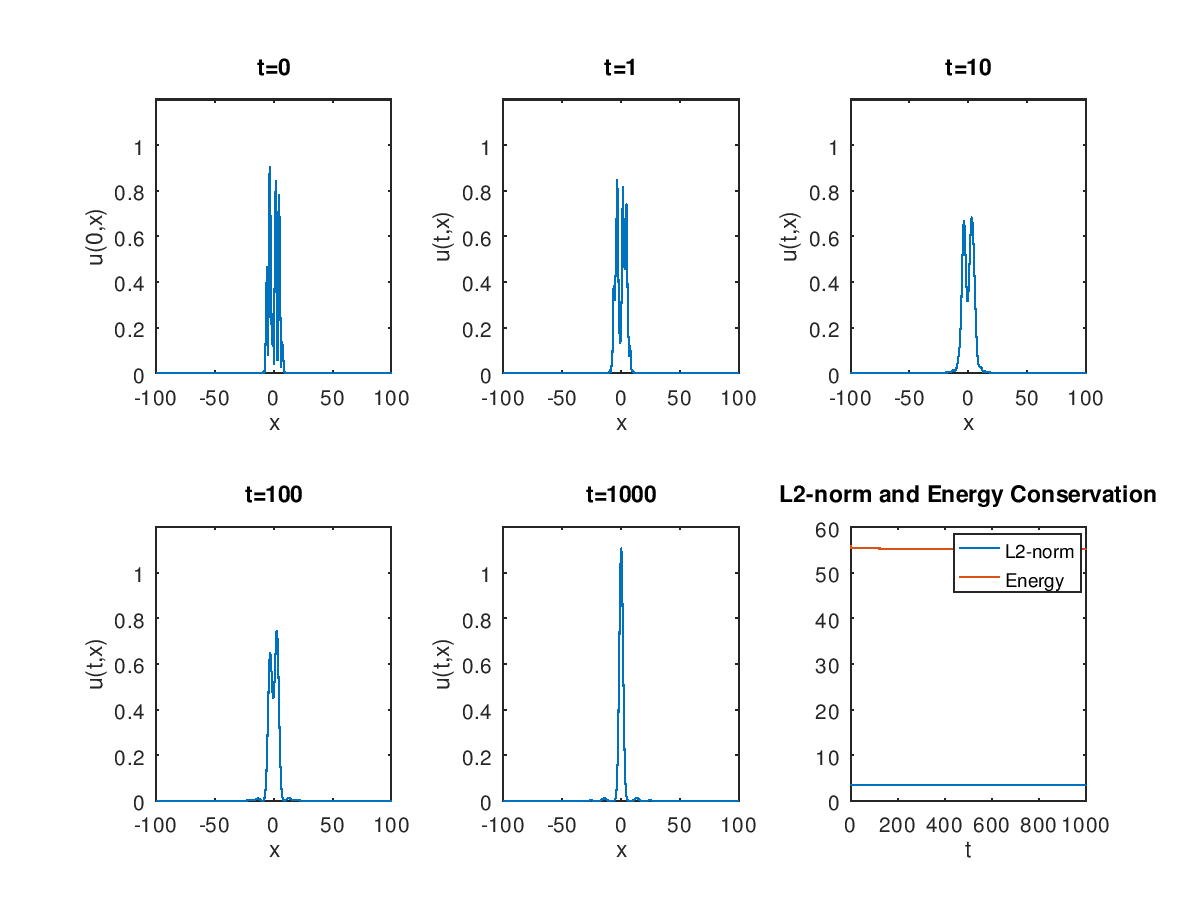}	
	\caption{Solution of equation \eqref{NMeq} with initial datum $\psi_0$ in the focusing case ($\lambda=-0.1$, $\mu=1$).}
	\label{fig:lambda_neg}
\end{figure}

\begin{figure}[p]
	\centering
		\includegraphics[width=0.85\textwidth,trim = 0cm 1cm 0cm 1cm, clip]{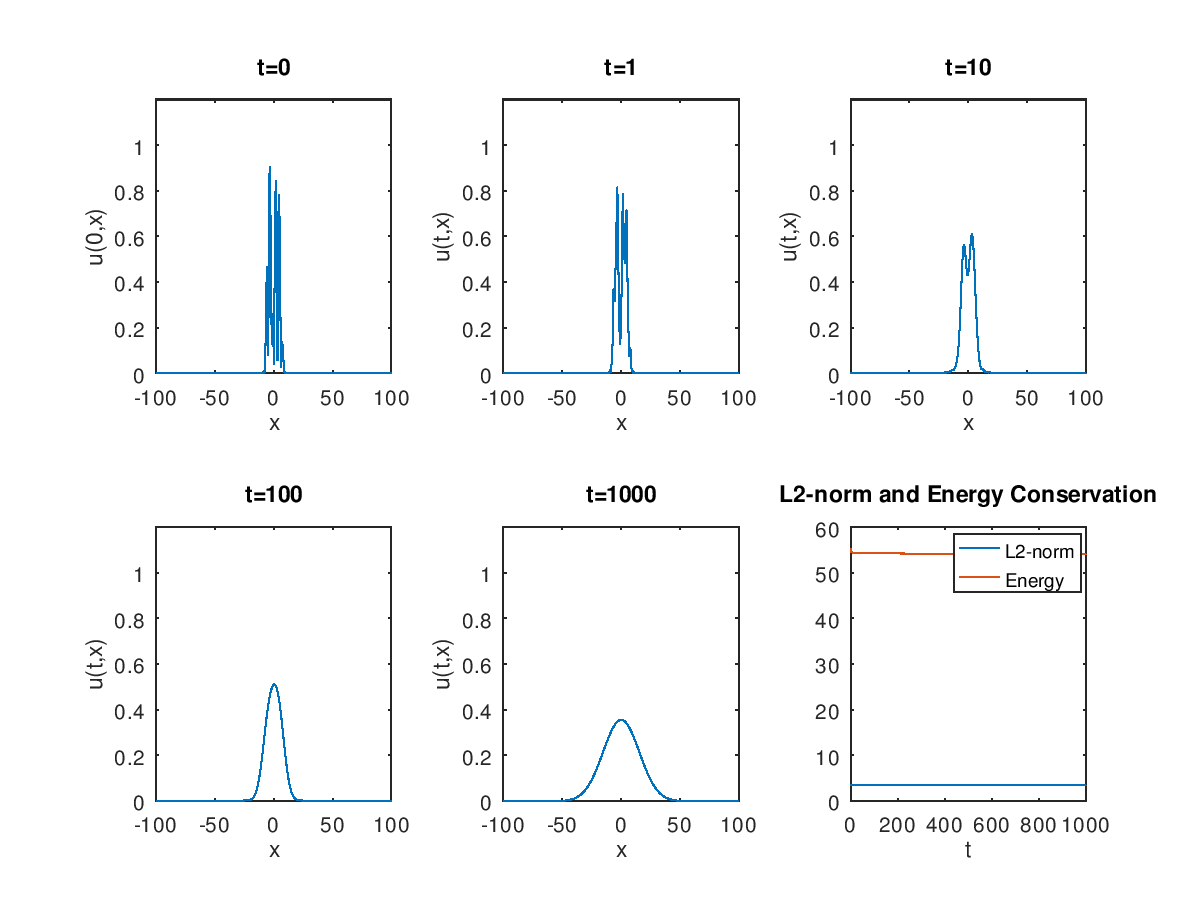}	
	\caption{Solution of equation \eqref{NMeq} with initial datum $\psi_0$ in the defocusing case ($\lambda=0.1$, $\mu=1$).}
	\label{fig:lambda_pos}
\end{figure}

\begin{figure}[p]
	\centering
		\includegraphics[width=0.85\textwidth,trim = 0cm 1cm 0cm 1cm, clip]{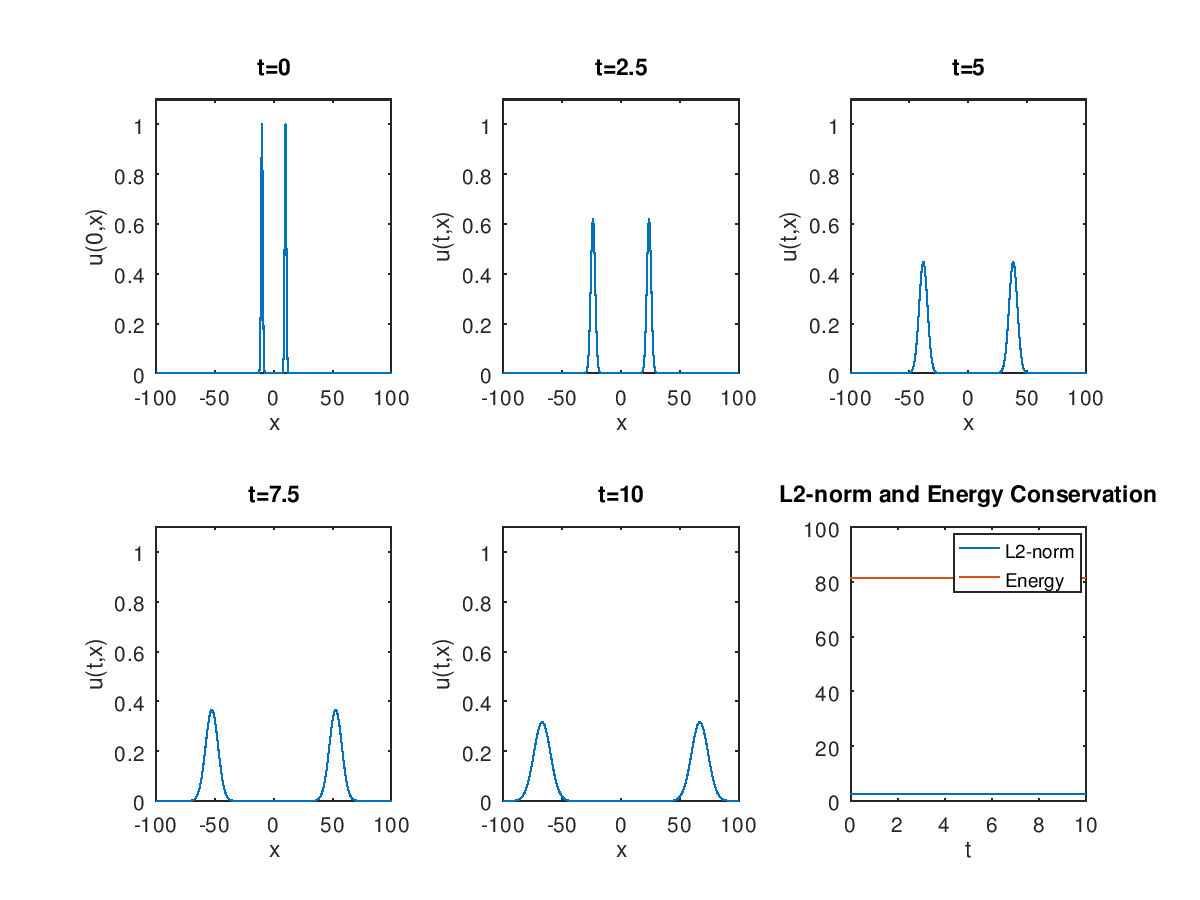}	
	\caption{Solution of equation \eqref{NMeq} with initial datum $\varphi_0$ in the free case ($\lambda=0$, $\mu=0$).}
	\label{fig:2_gaussian_free}
\end{figure}

\begin{figure}[p]
	\centering
		\includegraphics[width=0.85\textwidth,trim = 0cm 1cm 0cm 1cm, clip]{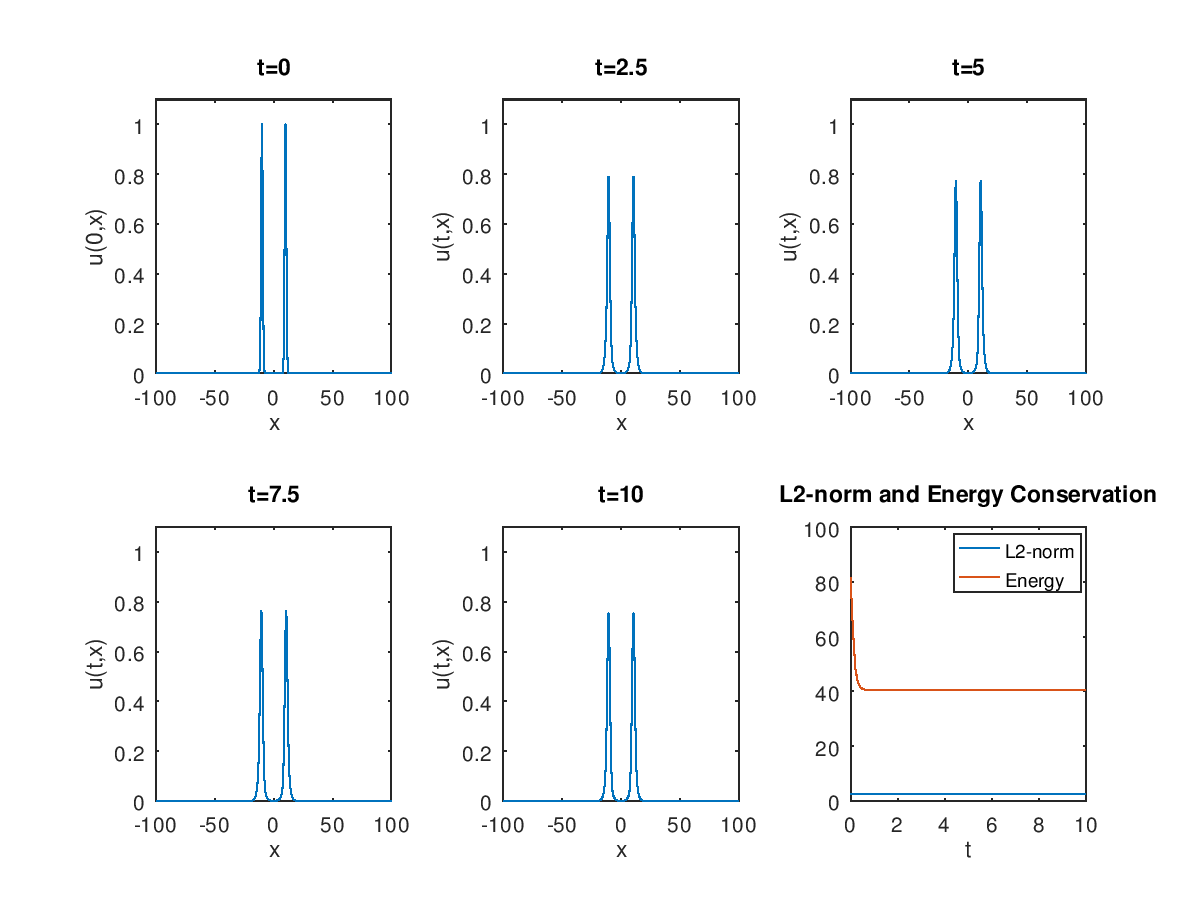}	
	\caption{Solution of equation \eqref{NMeq} with initial datum $\varphi_0$ in the focusing case ($\lambda=-0.1$, $\mu=10$).}
	\label{fig:2_gaussian_mu}
\end{figure}

\end{document}